\documentclass[11pt,a4paper,reqno]{article}
\usepackage{template}

\usepackage{cancel}

\usepackage{subcaption}

\DeclareMathOperator\Vol{Vol}
\DeclareMathOperator\sign{Sign}

\DeclareMathOperator\Per{Per}

\DeclareMathOperator\argmax{argmax}
\DeclareMathOperator\E{E}
\let\P\relax
\DeclareMathOperator\P{P}

\DeclareMathOperator{\normal}{Normal}

\DeclareMathOperator{\cov}{Cov}
\DeclareMathOperator\U{U}

\newcommand{\Z}{\mathbb{Z}}

\newcommand{\R}{\mathbb{R}}
\newcommand{\N}{\mathbb{N}}

\makeatletter
\newcommand\xleftrightarrow[2][]{%
  \ext@arrow 9999{\longleftrightarrowfill@}{#1}{#2}}
\newcommand\longleftrightarrowfill@{%
  \arrowfill@\leftarrow\relbar\rightarrow}
\makeatother
\newcounter{iconst}

\usetikzlibrary{patterns}

\newcommand{\extend}{\vcenter{\hbox{\;\tikz[scale=2.5]
      {
        \draw[color=black!60] (0, 0) rectangle (2ex, 2.5ex);
        \draw[color=black!60] (0, 0) rectangle (2.4ex, 2.9ex);
        \draw (0ex, .3ex) .. controls (1.2ex, 1ex) and (.8ex, 2ex) .. (2ex, 2.2ex);
        \draw[very thick] (2ex, 2.2ex) -- (2.4ex, 2.6ex);
        \draw [decorate,decoration={brace,amplitude=.3ex,raise=.3ex},yshift=0pt] (2ex, 0ex) -- (0ex, 0ex);
        \draw [decorate,decoration={brace,amplitude=.3ex,raise=.3ex},yshift=0pt] (0ex, 0ex) -- (0ex, 2.5ex);
        \draw [decorate,decoration={brace,amplitude=.3ex,raise=.3ex},yshift=0pt] (2.4ex, 2.9ex) -- (2.4ex, 2.5ex);
        \draw [decorate,decoration={brace,amplitude=.3ex,raise=.3ex},yshift=0pt] (2ex, 2.9ex) -- (2.4ex, 2.9ex);
        \draw (1ex, -.2ex) node[anchor=north] {$w(h) - 4$};
        \draw (2.2ex, 3ex) node[anchor=south] {$12$};
        \draw (2.5ex, 2.7ex) node[anchor=west] {$12$};
        \draw (-0.2ex, 1.25ex) node[anchor=east] {$3h/4$};

  }}}
}

\newcommand{\vdot}[3]{\vcenter{\hbox{\;\tikz[scale=1.2]
      {
        \fill[pattern=north west lines, pattern color=black!60] (2ex,0) rectangle ++(.2ex,4ex);
        \draw[color=black!60] (0, 0) rectangle (2ex, 4ex);
        \draw (1ex, 0) .. controls (0.4ex, .5ex) and (1ex, 1ex) .. (2ex, 1.3ex);
        \draw (2ex, 1.3ex) .. controls (1.4ex, 1.8ex) .. (2ex, 2.6ex);
        \draw (2ex, 2.6ex) .. controls (.8ex, 3.3ex) .. (1.5ex, 4ex);
        \draw (1ex,-.2ex) -- (1ex, .2ex);
        \draw[very thick, color=black] (0, 4ex) -- (2ex, 4ex);
        \draw (1ex, 0ex) node[anchor=north] {$\scriptstyle #1$};
        \draw (2ex, 2ex) node[anchor=west] {$\scriptstyle #2$};
        \draw (0ex, 0ex) node[anchor=east] {$\scriptstyle #3$};
        \fill (0, 0) circle (0.2ex);
  }}}
}

\newcommand{\vbar}[3]{\vcenter{\hbox{\;\tikz[scale=1.2]
      {
        \draw[color=black!60, dashed] (0, 0) -- (0,4ex) -- (2ex, 4ex);
        \draw[color=black!60] (0, 0) -- (2ex, 0);
        \draw (1ex, 0) .. controls (0.4ex, 2ex) and (1ex, 2.8ex) .. (2ex, 3ex);
        \draw (1ex,-.2ex) -- (1ex, .2ex);
        \draw[very thick, color=black] (2ex, 0) -- (2ex, 4ex);
        \draw (1ex, 0ex) node[anchor=north] {$\scriptstyle #1$};
        \draw (2ex, 2ex) node[anchor=west] {$\scriptstyle #2$};
        \draw (0ex, 0ex) node[anchor=east] {$\scriptstyle #3$};
        \fill (0, 0) circle (0.2ex);
  }}}
}

\newcommand{\vvdot}[4]{\vcenter{\hbox{\;\tikz[scale=1.2]
      {
        \fill[pattern=north west lines, pattern color=black!60] (2ex,0) rectangle ++(.2ex,4ex);
        \draw [decorate,decoration={brace,amplitude=.3ex,raise=.3ex},yshift=0pt] (1.2ex, 4ex) -- (2ex, 4ex);
        \draw[color=black!60] (0, 0) rectangle (2ex, 4ex);
        \draw (1ex, 0) .. controls (0.4ex, .5ex) and (1ex, 1ex) .. (2ex, 1.3ex);
        \draw (2ex, 1.3ex) .. controls (1.4ex, 1.8ex) .. (2ex, 2.6ex);
        \draw (2ex, 2.6ex) .. controls (.8ex, 3.3ex) .. (1.5ex, 4ex);
        \draw (1ex,-.2ex) -- (1ex, .2ex);
        \draw[very thick, color=black] (1.2ex, 4ex) -- (2ex, 4ex);
        \draw (1ex, 0ex) node[anchor=north] {$\scriptstyle #1$};
        \draw (1.7ex, 4ex) node[anchor=south] {$\scriptstyle #2$};
        \draw (2ex, 2ex) node[anchor=west] {$\scriptstyle #3$};
        \draw (0ex, 0ex) node[anchor=east] {$\scriptstyle #4$};
        \fill (0, 0) circle (0.2ex);
  }}}
}

\newcommand{\greenone}[4]{\vcenter{\hbox{\;\tikz[scale=1.2]
      {
        \fill[pattern=north west lines, pattern color=black!60] (2ex,0) rectangle ++(.2ex,4ex);
        \draw [decorate,decoration={brace,amplitude=.3ex,raise=.3ex},yshift=0pt] (1.2ex, 4ex) -- (2ex, 4ex);
        \draw[color=black!60] (0, 0) rectangle (2ex, 4ex);
        \draw (0ex, .3ex) .. controls (0.4ex, .5ex) and (1ex, 1ex) .. (2ex, 1.3ex);
        \draw (2ex, 1.3ex) .. controls (1.4ex, 1.8ex) .. (2ex, 2.6ex);
        \draw (2ex, 2.6ex) .. controls (.8ex, 3.3ex) .. (1.5ex, 4ex);
        \draw[very thick, color=black] (1.3ex, 4ex) -- (2ex, 4ex);
        \draw[very thick, color=black] (0, 0ex) -- (0ex, 4ex);
        \draw (1ex, 0ex) node[anchor=north] {$\scriptstyle #1$};
        \draw (2ex, 2ex) node[anchor=west] {$\scriptstyle #3$};
        \draw (1.7ex, 4ex) node[anchor=south] {$\scriptstyle #2$};
        \draw (0ex, 0ex) node[anchor=east] {$\scriptstyle #4$};
        \fill (0, 0) circle (0.2ex);
  }}}
}

\newcommand{\neasttwo}[4]{\vcenter{\hbox{\;\tikz[scale=1.2]
      {
        \draw[color=black!60] (0, 0) rectangle (2ex, 4ex);
        \draw [decorate,decoration={brace,amplitude=.3ex,raise=.3ex},yshift=0pt] (1.2ex, 4ex) -- (2ex, 4ex);
        \draw (0ex, 1.6ex) .. controls (1.4ex, 2.2ex) and (.6ex, 3.4ex) .. (1.8ex, 4ex);
        \draw[very thick, color=black] (0, 0) -- (0, 4ex);
        \draw[very thick, color=black] (1.3ex, 4ex) -- (2ex, 4ex);
        \draw (1ex, 0ex) node[anchor=north] {$\scriptstyle #1$};
        \draw (1.5ex, 4ex) node[anchor=south] {$\scriptstyle #2$};
        \draw (2ex, 2ex) node[anchor=west] {$\scriptstyle #3$};
        \draw (0ex, 0ex) node[anchor=east] {$\scriptstyle #4$};
        \fill (0, 0) circle (0.2ex);
  }}}
}

\newcommand{\neasttworef}[4]{\vcenter{\hbox{\;\tikz[scale=1.2]
      {
        \draw[color=black!60] (0, 0) rectangle (2ex, 4ex);
        \draw [decorate,decoration={brace,amplitude=.3ex,raise=.3ex},yshift=0pt] (0ex, 4ex) -- (0.8ex, 4ex);
        \draw (2ex, 1.6ex) .. controls (.6ex, 2.2ex) and (1.4ex, 3.4ex) .. (.2ex, 4ex);
        \draw[very thick, color=black] (2ex, 0) -- (2ex, 4ex);
        \draw[very thick, color=black] (0.7ex, 4ex) -- (0ex, 4ex);
        \draw (1ex, 0ex) node[anchor=north] {$\scriptstyle #1$};
        \draw (.5ex, 4ex) node[anchor=south] {$\scriptstyle #2$};
        \draw (2ex, 2ex) node[anchor=west] {$\scriptstyle #3$};
        \draw (0ex, 0ex) node[anchor=east] {$\scriptstyle #4$};
        \fill (0, 0) circle (0.2ex);
  }}}
}

\newcommand{\vdotdotgood}[3]{\vcenter{\hbox{\;\tikz[scale=1.2]
      {
        \draw[color=black!60, dashed] (0, 0) -- (0,4ex) -- (2ex, 4ex) -- (2ex, 0);
        \draw[color=black!60] (0, 0) -- (2ex, 0);
        \draw (1ex, 0) .. controls (-1ex, 1ex) and (3ex, 2ex) .. (.3ex, 4ex);
        \draw (1ex,-.2ex) -- (1ex, .2ex);
        \draw[very thick, color=black] (0, 4ex) -- (2ex, 4ex);
        \draw (1ex, 0ex) node[anchor=north] {$\scriptstyle #1$};
        \draw (2ex, 2ex) node[anchor=west] {$\scriptstyle #2$};
        \draw (0ex, 0ex) node[anchor=east] {$\scriptstyle #3$};
        \fill (0, 0) circle (0.2ex);
      }}}
}

\newcommand{\vdotdot}[4]{\vcenter{\hbox{\;\tikz[scale=1.2]
      {
        \fill[pattern=north west lines, pattern color=black!60] (2ex,0) rectangle ++(.2ex,4ex);
        \draw [decorate,decoration={brace,amplitude=.3ex,raise=.3ex},yshift=0pt] (0ex, 4ex) -- (.7ex, 4ex);
        \draw[color=black!60] (0, 0) rectangle (2ex, 4ex);
        \draw (1ex, 0) .. controls (0.4ex, .5ex) and (1ex, 1ex) .. (2ex, 1.3ex);
        \draw (2ex, 1.3ex) .. controls (1.4ex, 1.8ex) .. (2ex, 2.6ex);
        \draw (2ex, 2.6ex) .. controls (.8ex, 3.3ex) .. (.4ex, 4ex);
        \draw (1ex,-.2ex) -- (1ex, .2ex);
        \draw[very thick, color=black] (0, 4ex) -- (.7ex, 4ex);
        \draw (1ex, 0ex) node[anchor=north] {$\scriptstyle #1$};
        \draw (.35ex, 4ex) node[anchor=south] {$\scriptstyle #2$};
        \draw (2ex, 2ex) node[anchor=west] {$\scriptstyle #3$};
        \draw (0ex, 0ex) node[anchor=east] {$\scriptstyle #4$};
        \fill (0, 0) circle (0.2ex);
      }}}
}

\newcommand{\vdotdotprime}[4]{\vcenter{\hbox{\;\tikz[scale=1.2]
      {
        \fill[pattern=north west lines, pattern color=black!60] (2ex,0) rectangle ++(.2ex,4ex);
        \draw [decorate,decoration={brace,amplitude=.3ex,raise=.3ex},yshift=0pt] (0ex, 4ex) -- (.7ex, 4ex);
        \draw[color=black!60] (0, 0) rectangle (2ex, 4ex);
        \draw (1ex, 0) .. controls (0.4ex, .5ex) and (1ex, 1ex) .. (2ex, 1.3ex);
        \draw (2ex, 1.3ex) .. controls (1.4ex, 1.8ex) .. (2ex, 2.6ex);
        \draw (2ex, 2.6ex) .. controls (1ex, 3.3ex) .. (1.5ex, 4ex);
        \draw (1ex,-.2ex) -- (1ex, .2ex);
        \draw[very thick, color=black] (0, 4ex) -- (.7ex, 4ex);
        \draw (1ex, 0ex) node[anchor=north] {$\scriptstyle #1$};
        \draw (.35ex, 4ex) node[anchor=south] {$\scriptstyle #2$};
        \draw (2ex, 2ex) node[anchor=west] {$\scriptstyle #3$};
        \draw (0ex, 0ex) node[anchor=east] {$\scriptstyle #4$};
        \fill (0, 0) circle (0.2ex);
      }}}
}

\newcommand{\vdotdotdot}[4]{\vcenter{\hbox{\;\tikz[scale=1.2]
      {
        \draw [decorate,decoration={brace,amplitude=.3ex,raise=.3ex},yshift=0pt] (0ex, 4ex) -- (.7ex, 4ex);
        \draw[color=black!60] (0, 0) rectangle (2ex, 4ex);
        \draw (1ex, 0) .. controls (-1ex, 1ex) and (3ex, 2ex) .. (.3ex, 4ex);
        \draw[very thick, color=black] (0, 4ex) -- (.7ex, 4ex);
        \draw (1ex,-.2ex) -- (1ex, .2ex);
        \draw (1ex, 0ex) node[anchor=north] {$\scriptstyle #1$};
        \draw (.35ex, 4ex) node[anchor=south] {$\scriptstyle #2$};
        \draw (2ex, 2ex) node[anchor=west] {$\scriptstyle #3$};
        \draw (0ex, 0ex) node[anchor=east] {$\scriptstyle #4$};
        \fill (0, 0) circle (0.2ex);
      }}}
}

\newcommand{\vdotdotdotref}[4]{\vcenter{\hbox{\;\tikz[scale=1.2]
      {
        \draw [decorate,decoration={brace,amplitude=.3ex,raise=.3ex},yshift=0pt] (1.3ex, 4ex) -- (2ex, 4ex);
        \draw[color=black!60] (0, 0) rectangle (2ex, 4ex);
        \draw (1ex, 0) .. controls (3ex, 1ex) and (-1ex, 2ex) .. (1.7ex, 4ex);
        \draw[very thick, color=black] (1.3ex, 4ex) -- (2ex, 4ex);
        \draw (1ex,-.2ex) -- (1ex, .2ex);
        \draw (1ex, 0ex) node[anchor=north] {$\scriptstyle #1$};
        \draw (1.65ex, 4ex) node[anchor=south] {$\scriptstyle #2$};
        \draw (2ex, 2ex) node[anchor=west] {$\scriptstyle #3$};
        \draw (0ex, 0ex) node[anchor=east] {$\scriptstyle #4$};
        \fill (0, 0) circle (0.2ex);
      }}}
}

\begin{document}

\title{Fluctuation bounds for symmetric random walks on dynamic environments via Russo-Seymour-Welsh}

\date{\today}
\author{
  Rangel Baldasso
  \thanks{Email: \ \texttt{rangel@mat.puc-rio.br}; \ Department of Mathematics, PUC-Rio, Rua Marqu\^es de S\~ao Vicente 225, G\'avea, 22451-900 Rio de Janeiro, RJ - Brazil.}
  \and
  Marcelo Hil\'ario
  \thanks{Email: \ \texttt{mhilario@mat.ufmg.com}; \ Departement of Mathematics, Universidade Federal de Minas Gerais, Av.\ Ant\^onio Carlos 6627, 31270-901 Belo Horizonte, MG - Brazil.}
  \and
  Daniel Kious
  \thanks{Email: \ \texttt{d.kious@bath.ac.uk}; \ Department of Mathematical Sciences, University of Bath, Claverton Down, BA2 7AY Bath, UK.}
  \and
  Augusto Teixeira
  \thanks{Email: \ \texttt{augusto@impa.br}; \ IMPA, Estrada Dona Castorina 110, 22460-320 Rio de Janeiro, RJ - Brazil and IST, University of Lisbon, Portugal.}
}

\maketitle

\begin{abstract}
  In this article we prove a lower bound for the fluctuations of symmetric random walks on dynamic random environments in dimension $1 + 1$ in the perturbative regime where the walker is weakly influenced by the environment.
  We suppose that the random environment is invariant with respect to translations and reflections, satisfy the FKG inequality and a mild mixing condition.
  The techniques employed are inspired by percolation theory, including a Russo-Seymour-Welsh (RSW) inequality.
  To exemplify the generality of our results, we provide two families of fields that satisfy our hypotheses: a class of Gaussian fields and Confetti percolation models.

  \medskip


\end{abstract}

\section{Introduction}
\label{s:intro}
~
\par Random walks in random environments have been intensively studied over the last decades.
Some early works were motivated by applications in biophysics, chemistry and physics \cite{chernov67, temkin69}.
From the theoretical point of view, they are sources of challenging mathematical problems and some basic questions regarding the asymptotic behavior of such processes remain open to this date.
In this paper we will focus on random walks on dynamic random environments (RWDRE), a class of random walks whose transition probabilities depend on a random environment that also evolves stochastically in time.

There is currently a solid understanding on the asymptotic behavior of such processes under the assumption that the dynamic random environment mixes fast and uniformly.
For instance, in that context, methods like renewal arguments \cite{comets2004, AvenaThesis} and Markov techniques \cite{redig2013random} have been employed to successfully derive Laws of Large Numbers (LLN) and Central Limit Theorems (CLT) in a great degree of generality.

In the case of non-uniformly mixing environments, progress has been made in a few fronts as well \cite{mountford2015,Avena2017,10.1214/19-AOP1414}, even for models that are particularly difficult to deal with due to the conservative nature of the underlying environment \cite{zbMATH06514478,BHdSST18b,BHdSST18,berard2016fluctuations,huveneers2015random,HKT19} by employing various techniques such as renewal theory, Markov theory, spectral analysis, and renormalization.
As a rule, results for RWRE on non-uniformly mixing environment feature one or more of the following limitations: they are model specific \cite{mountford2015}, do not provide any information on fluctuations \cite{10.1214/19-AOP1414}, are perturbative in nature \cite{zbMATH06514478,huveneers2015random} or require exponential mixing \cite{Avena2017,mountford2015}.
These shortcomings indicate that there are still fundamental questions that remain unanswered in RWDRE.
For instance, the question of whether trapping effects are relevant to the point that these models can exhibit anomalous fluctuations seems to be a very interesting open problem.
We discuss the previous works on RWDRE in more detail after the statement of our main results.

\bigskip

The main objective of this article is to present a new approach that leads to lower bounds on the fluctuations of symmetric random walks on top of random environments that have slow and non-uniform mixing.

Our arguments take inspiration from percolation theory as in \cite{10.1214/19-AOP1414, HKT19}, but with an extra input inspired by \cite{D-CTT18}, where a Russo-Seymour-Welsh (RSW) type argument was introduced to study oriented percolation.

\bigskip

Let us now introduce the setup of our main result.
For this we start with a random environment $f: \mathbb{Z}^2 \to \mathbb{R}$ sampled according to a certain probability measure $\P$.
Given such a random environment $f$ and fixing $\delta \in [0,1/2]$, we build a discrete-time random walk $(X_n)_{n \geq 0}$ evolving on top of $f$.
For that we fix $X_0 = 0$ almost surely and define a Markov dynamics in $\Z$ with transition probabilities
\begin{equation}
  \label{e:X_n}
  \begin{split}
    \mathbb{P}^{f} \big( X_{n + 1} = x + 1 \big| X_n = x \big)
    & = 1 - \mathbb{P}^{f} \big( X_{n + 1} = x - 1 \big| X_n = x \big)\\
    & =
    \begin{cases}
      \frac{1}{2} + \delta, & \text{if $f(x, n) > 0$}\\
      \frac{1}{2}, & \text{if $f(x, n) = 0$}\\
      \frac{1}{2} - \delta, & \text{if $f(x, n) < 0$}
    \end{cases}\\
    & = \frac{1}{2} + \delta \big( \textbf{1}_{f_{(x, n)} > 0} - \textbf{1}_{f_{(x, n)} < 0} \big).
  \end{split}
\end{equation}
Roughly speaking, the walker only observes the sign of $f$ at it current position in order to determine its next step.

Our main result concerns lower bounds on the fluctuations of the above random walk.
But before stating it, let us list the three hypotheses that we require on the environment $f$.
These can be informally described as: \emph{symmetry}, {\emph{translation invariance}}, \emph{positive association}, and \emph{decoupling}.

We start by assuming that the environment is invariant with respect to translations and to reflections about the vertical axis.
More precisely, suppose that for every $(x_0, y_0) \in \mathbb{Z}^2$,
\begin{equation}
  \label{e:translation}
  \tag{T}
  \big( f(x, y) \big)_{(x, y) \in \mathbb{Z}^2} \overset{d}{\sim} \big( f(x, y) \big)_{(x + x_0, y + y_0) \in \mathbb{Z}^2}
\end{equation}
and
\begin{equation}
  \label{e:symmetry}
  \tag{R}
  \big( f(x, y) \big)_{(x, y) \in \mathbb{Z}^2} \overset{d}{\sim} \big( f(-x, y) \big)_{(x, y) \in \mathbb{Z}^2}.
\end{equation}
under the law $\P$.

We will also require that $f$ satisfies a Harris-FKG inequality.
An event $A$ is said increasing if
\begin{equation}\label{eq:increasing}
f \leq \tilde{f} \text{ and } f \in A \quad \text{implies} \quad \tilde{f} \in A.
\end{equation}
We suppose that for any increasing events $A$ and $B$
\begin{equation}
  \label{e:fkg}
  \tag{FKG}
  \P \big( f \in A \cap B\big) \geq \P \big( f \in A \big) \P \big( f \in B \big).
\end{equation}
In fact we will only use the inequality above for events that depend on finitely many coordinates of $f$.

Finally we will assume that $f$ satisfies a decoupling condition.
Informally speaking, it states that with high probability $f$ can be approximated inside a rectangle by another field with finite range of dependence.
Let $\d$ denote the $L^{1}$-distance in $\Z^{2}$.

\begin{definition}[Decoupling condition]
  \label{def:decoupling}
  We say that a field $f$ satisfies the \emph{decoupling condition} with decay rate $\varepsilon(\cdot)$ if the following holds. For every integer $r \geq 2$ and every rectangle $C = [a, a+w] \times [b, b+h] \subset \Z^{2}$ with $w, h \geq 1$, there exists a coupling between $f$ and a field ${f}^{C, r}$ such that
  \begin{equation}
    \label{e:likely_equal}
    \P \big( {f}^{C, r} \neq f \big) \leq \varepsilon( w, h, r),
  \end{equation}
  and, given any two sets $A \subset C$ and $B \subset \Z^{2}$,
  \begin{equation}
    \label{e:finite_range}
  \text{if $\d \big(A, B \big) > r$, then the fields ${f}^{C, r} \big|_{A}$ and ${f}^{C, r} \big|_{B}$ are independent.}
  \end{equation}
\end{definition}
We have not found this definition in the existing literature and we believe it to be interesting on its own.

We are now in the position to state our main result.

\nc{c:alpha}
\nc{c:decoupling}
\nc{c:perturbative}
\begin{theorem}
  \label{t:main}
   There exists a constant $\uc{c:alpha} > 0$ such that the following holds.
  Suppose that the random field $f$ satisfies \eqref{e:translation}, \eqref{e:symmetry}, \eqref{e:fkg}, and Definition~\ref{def:decoupling} with
  \begin{equation}
    \label{e:polynomial}
    \varepsilon(w, h, r) \leq \uc{c:decoupling} \big( wh + (w + h)r + r^2 \big) r^{-\uc{c:alpha}},
  \end{equation}
  for some $\uc{c:decoupling} > 0$.
  There exist $\uc{c:perturbative}>0$ and $\xi > 0$ such that if $\delta \leq \uc{c:perturbative}$ then the random walk $(X_n)_{n \in \N}$ introduced in \eqref{e:X_n} satisfies
  \begin{equation}
    \label{e:main}
    \inf_{n \geq 1} \mathbb{P} \big( |X_n| \geq n^\xi \big) > 0.
  \end{equation}
\end{theorem}

\begin{remark}
  We observe that under the condition that $f$ satisfies the decoupling condition as in Definition~\ref{def:decoupling}, one can obtain a Law of Large Numbers for the above random walk using \cite{10.1214/19-AOP1414} (see also \cite{HKT19}).
  Indeed, we prove in Remark~\ref{r:lln} that, if $\uc{c:alpha} > 10$, then $\lim_{n \to \infty} X_n/n = 0$, for almost every realization of the environment and the random walk.
\end{remark}

\begin{remark}
  An important limitation of Theorem~\ref{t:main} concerns the perturbative nature of the interaction between the random walk and the environment.
  This can be seen by the fact that $\delta$ is required to be small.
  The same limitation is also present in other works on RWDRE, such as \cite{Avena2017,2012arXiv1209.1511D}.
\end{remark}

\bigskip

For the sake of concreteness, let us give two examples of random environments that satisfy our hypotheses.

\nc{c:correlation_decay}
\nc{c:correlation_convolution}

The first of them is a Gaussian field.
Fix first a function $q: \mathbb{Z}^2 \to \mathbb{R}_+$, with $q(o) > 0$ (where $o$ stands for the origin) that is symmetric under reflection about the vertical axis, that is,
\begin{equation}
\label{e:symmetry_q}
q(x_{1}, x_{2}) = q(-x_{1}, x_{2}),
\end{equation}
and suppose that it satisfies the bound
\begin{equation}
  \label{e:q}
  q(x) \leq \uc{c:correlation_decay} |x|^{-\beta}, \text{ for all $x \in \mathbb{Z}^2 \setminus \{o\}$},
\end{equation}
for some $\uc{c:correlation_decay}, \beta > 2$.

Consider now a Gaussian Field $g$ on $\mathbb{Z}^2$ whose covariance structure is given by
\begin{equation}
  \label{e:covariance}
  \Cov(g_x, g_y) = q * q(x - y) \leq \uc{c:correlation_convolution} |x - y|^{-\beta + 2},
\end{equation}
which is well defined as soon as we assume $q \in \ell^{2}(\Z^{2})$ (see \eqref{eq:construction_gf} for a precise construction of this field). See Figure~\ref{f:gf} for an illustration.

\begin{figure}[ht]
  \centering
  \includegraphics[trim=30 40 100 150,clip,width=.5\textwidth]{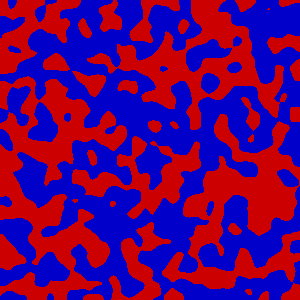}
  \caption{A simulation of the environment $f$ with  $q(x) = \exp\{-|x|^2/5\}$.
    Above $\{f > 0\}$ is represented in red, while $\{f \leq 0\}$ is represented in blue.}
   \label{f:gf}
\end{figure}

Perhaps the main example of Gaussian field satisfying these hypotheses is the Bargmann-Fock, which is obtained by choosing $q(x) = \big(\frac{2}{\pi}\big)^{\frac{1}{2}}e^{-|x|^{2}}$. Another example is the rational quadratic kernel $q(x)=(1+|x|^{2})^{-\frac{\alpha}{2}}$, for $\alpha > \frac{1}{2}$.
Let us finally mention that, although the $d$-dimensional discrete Gaussian Free Field (GFF) does not fall into the setting described above, our results also hold for a GFF on $\mathbb{Z}^d$ restricted to the two-dimensional Euclidean space as long as $d$ is large enough, see Remark~\ref{r:gff}.

This random environment is defined in more detail in Section~\ref{s:examples}, where we also prove that it satisfies the decoupling condition with decay as in \eqref{e:polynomial} with $\uc{c:alpha} = \beta-\tfrac{3}{2}$.
As a consequence we obtain the following corollary.

\begin{corollary}
  \label{t:gf}
 Suppose that $g$ is given by \eqref{e:covariance} with $q$ satisfying \eqref{e:q} with $\beta \geq \uc{c:alpha} + 3/2$.
  Then there exist $\uc{c:perturbative}, \xi > 0$ such that, if $\delta \leq \uc{c:perturbative}$, the random walk $ (X_n)_{n \in \N}$ introduced in \eqref{e:X_n} satisfies
  \begin{equation}
    \label{e:gf}
    \inf_{n \geq 1} \mathbb{P} \big( |X_n| \geq n^\xi \big) > 0.
  \end{equation}
\end{corollary}

\bigskip

The second environment which we consider is given by a random coloring of the plane induced by a collection of overlapping balls of random radii.
It may be regarded as a symmetric version of the Boolean percolation, also inspired by Confetti Percolation \cite{hirsch2015harris} and the Dead Leaves Model \cite{bordenave2006dead,jeulin2021dead}.
Let us first provide a brief, informal description.

\nc{c:decay_nu}
Fix a probability measure $\nu$ on $\mathbb{R}_+$ for which there exists a positive constant $\uc{c:decay_nu}$ such that
\begin{equation}
  \label{e:nu_decay}
  \nu\big( [r, \infty) \big) \leq \uc{c:decay_nu} r^{-\alpha}, \text{ for all } r \geq 0,
\end{equation}
with some $\alpha > 2$.
The measure $\nu$ will control the radius distribution of the balls composing our environment and consequently its range of dependence.
We will introduce a Poisson Point Process $(x_i)_{i \geq 1}$ on $\mathbb{R}^2$ with density one, and place a ball centered at each point $x_i$ with independent radius $r_i$ distributed according to $\nu$.
We then independently assign to each ball one of two colors; \emph{blue} (represented by $-1$) or \emph{red} (represented by $1$) with probability $1/2$ each.
Next we define a coloring of the whole plane $f: \mathbb{R}^2 \to \{-1, 0, 1\}$ through the following rule: if a point $x \in \mathbb{R}^2$ is not contained in any of the above balls, we paint it \emph{gray} (which will be encoded by $0$), otherwise we paint it with the color of a ball chosen uniformly among all balls containing $x$.

\definecolor{mypurple}{rgb}{.66,.27,.74}
\definecolor{myred}{rgb}{.86,.3,.46}
\definecolor{myblue}{rgb}{.36,.29,.89}

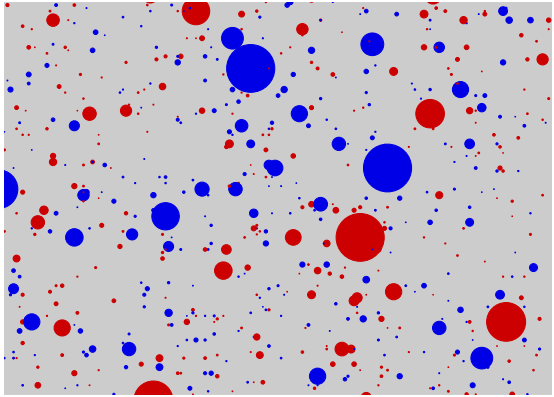
\begin{figure}[ht]
  \centering
  \begin{tikzpicture}[scale=2]
    \begin{scope}
      \clip (-.6, -.6) rectangle (3, 2);
      \draw[color=gray!40!white,fill=gray!40!white] (-.6, -.6) rectangle (3, 2);
      \foreach \z in {
36, 64, 52, 366, 30, 22, 25, 184, 86, 147, 11, 76, 6, 22, 3, 105, 141, 34, 16, 62, 22, 478, 57, 29, 682, 4, 20, 42, 35, 140, 37, 3, 61, 21, 36, 128, 47, 25, 13, 26, 11, 8, 11, 48, 17, 25, 71, 0, 16, 250, 301, 25, 34, 36, 19, 189, 86, 16, 2, 4, 47, 23, 19, 35, 19, 32, 135, 31, 391, 52, 61, 21, 7, 106, 31, 289, 43, 68, 463, 46, 89, 160, 415, 123, 31, 30, 1, 197, 66, 277, 201, 53, 74, 57, 140, 4, 31, 96, 34, 7, 18, 29, 3, 66, 21, 861, 10, 94, 144, 119, 18, 33, 100, 23, 10, 220, 11, 144, 133, 92, 40, 826, 2, 2, 8, 7, 519, 0, 15, 8, 110, 32, 66, 173, 5, 154, 1, 105, 98, 33, 138, 16, 90, 10, 30, 67, 11, 49, 28, 5, 26, 650, 82, 29, 10, 180, 225, 37, 58, 111, 4, 86, 117, 12, 8, 98, 64, 75, 4, 15, 121, 46, 41, 330, 124, 5, 19, 35, 40, 54, 0, 393, 25, 602, 54, 15, 240, 147, 29, 12, 88, 81, 25, 15, 2, 4, 73, 18, 13, 1461, 43, 480, 1, 152, 15, 19, 47, 222, 19, 172, 42, 28, 32, 33, 14, 24, 30, 41, 91, 7, 100, 61, 224, 81, 178, 0, 105, 204, 98, 190, 324, 47, 93, 26, 27, 1180, 10, 40, 51, 55, 48, 5, 14, 285, 31, 22, 57, 2, 32, 52, 32, 2, 24, 54, 1, 9, 97, 50, 4, 100, 1, 28, 14, 75, 51, 16, 403, 1, 11
      } {
      \foreach \e in {1,2} {
       \pgfmathrandominteger{\x}{0}{200};
        \pgfmathrandominteger{\y}{0}{200};
        \pgfmathrandominteger{\c}{1}{1000};
        \pgfmathrandominteger{\d}{0}{1};

        \pgfmathsetmacro{\radius}{.0001*(\z+20)}%

        \draw[color=blue!90!black,fill=blue!90!black]
                  ({.02*\x-1}, {.02*\y-1}) circle ({\radius});

        \pgfmathrandominteger{\x}{0}{200};
        \pgfmathrandominteger{\y}{0}{200};
        \pgfmathrandominteger{\c}{1}{1000};
        \pgfmathrandominteger{\d}{0}{1};
        %\pgfmathsetmacro{\radius}{.0006 * \z}%

        \draw[color=red!80!black,fill=red!80!black]
                  ({.02*\x-1}, {.02*\y-1}) circle ({\radius});
}}
    \end{scope}
  \end{tikzpicture}
  \caption{An illustration of the environment $f$, where the radius distribution is given by $R = \exp\{Z\}$, $Z \sim \Exp(2)$.}
   \label{f:confetti}
\end{figure}

In Section~\ref{s:examples} we also prove that the Confetti random environment also satisfies the hypotheses of Theorem~\ref{t:main} with
\begin{equation}
  \varepsilon(w, h, r) = c(wh + (w + h)r + r^2)r^{-\alpha},
\end{equation}
implying the following result.

\begin{corollary}
  \label{t:confetti}
  Suppose that $f$ is the coloring given by the Confetti random environment defined above, for $\alpha$ large enough.
  Then there exist $\uc{c:perturbative}, \xi > 0$ such that if $\delta \leq \uc{c:perturbative}$,
  \begin{equation}
    \inf_{n \geq 1} \mathbb{P} \big( |X_n| \geq n^\xi \big) > 0.
  \end{equation}
\end{corollary}

\paragraph{Previous works.}

When a random walk evolves on top of a static random environment, different phenomena emerge, depending on the regime considered. In fact, if the random walk turns out to be recurrent, it exhibits anomalous fluctuations of order $\log^{2}n$, as verified in~\cite{sinai_clt, kesten_limit_rwsre}. Under additional hypotheses,~\cite{kks} verifies that the random walk has Gaussian fluctuations in the transient phase.

If the environment is Markovian and features a positive spectral gap in $\textnormal{L}^2$, analytical techniques can be successfully used to study perturbative random walks, see \cite{Avena2017}.
Some very robust results can also be obtained in a model-by-model basis using techniques that rely on particular properties of the environment, such as \cite{mountford2015,berard2016fluctuations,zbMATH06514478}.
Some techniques can be applied in general for environments that mix polynomially but non-uniformly, such as in~\cite{10.1214/19-AOP1414}.
However they only provide a law of large numbers for the random walk.
Let us also mention the works~\cite{abf, oriane_antisymmetry}, where the relation $-v(\delta) = v(-\delta)$ for the asymptotic speed of random walk defined in~\eqref{e:X_n} with $\delta$ and $-\delta$ is established for a large class of reversible environments, without assumptions on invariance under space reflections.

The case when the environment is a symmetric exclusion process in equlibrium has received great attention.
Under appropriate rescaling of space and time,~\cite{afjv} characterizes the hydrodynamic limit for such walks.
Fluctuation bounds and large deviation principles for these rescaled processes were then obtained in~\cite{jm, ajv}. Without any rescaling,~\cite{HKT19} proves that, for each two fixed nearest-neighbor transition probabilities of the random walk, there exist at most two possible densities of the exclusion process for which the random walk does not satisfy a law of large numbers.

A particular type of symmetry hypothesis poses significant challenge to the study of these models.
If the annealed law of the random walk and the environment is ergodic and invariant with respect to reflections over the space direction, the random walk fails to present ballistic behavior (see Remark~\ref{r:recurrent}), a key assumption explored in previous works~\cite{zbMATH06514478, huveneers2015random}.
The study of symmetric random walks on top of environments given by conservative particle systems emphasizes the poor mixing properties of these environments to such an extent that, to the best of our knowledge, no rigorous results on the fluctuations of such random walks in the so-called non-nestling case are known.
In the particular case of balanced random walks, we remark that a quenched central limit theorem was proved in~\cite{dgr}.
It is striking that for such models even simulations seem to reach conflicting predictions on whether such random walks are diffusive \cite{avena2012continuity,huveneers2018response}.

Random walks in cooling random environment were introduced in~\cite{avena_denhollander} as an interpolation between the static and dynamical cases.
In this model, a sequence of deterministic refreshing times is introduced, the environment is independently resampled at these times and kept constant between two resamplings.
The behavior of the random walk at the level of fluctuations is highly influenced by how fast the growth of these refreshing times is.
If the increments of these times do not diverge, Gaussian fluctuations are observed in~\cite{avena_denhollander}.
In the case where these increments grow polynomially or exponentially,~\cite{avena_denhollander} still verifies the existence of Gaussian fluctuations but with no diffusive scaling.
In this last case,~\cite{rwcre4} strengthens the results to a functional central limit theorem, with limit given by a time modification of a Brownian motion.
The works \cite{rwcre3, rwcre5} consider different types of increment growth for the refreshing times and observe a different plethora of phenomena regarding the fluctuation regimes of these random talks, ranging from pure Gaussian fluctuations to mixtures of stable laws.

Finally, let us mention~\cite{huveneers2020evolution}, where the behavior of a random walk driven by a random environment given as the time evolution of the heat equation with initial condition given by a one-dimensional Brownian motion. In this case, the random walk presents Gaussian subdiffusive behavior in the short time limit and diffusive behavior in the long time limit.

\paragraph{Overview of the proofs.}

The main results and techniques employed in the paper are inspired by renormalization arguments from percolation theory.
Using percolation techniques to study random walks on dynamic random environments is not new, see \cite{zbMATH06514478,10.1214/19-AOP1414,HKT19} for some examples of this interplay.
The novelty that appeared in this article was to use Russo-Seymour-Welsh type arguments in order to control the fluctuations of the random walk.

The first step of the proof is to prove a decoupling bound for the random environment, see Definition~\ref{def:decoupling} and Lemmas~\ref{l:error_decoupling} and~\ref{l:error_decoupling_confetti}.
The decoupling introduced in Definition~\ref{def:decoupling} seems to be new and of independent interest.

Having established this decoupling, we introduce a graphical construction of the random walk that allows us to start many walks from different space-time points in the plane, in such a way that their trajectories cannot cross each other.
With this graphical construction we are able to define events that are very similar to crossing events in oriented percolation models inspired by \cite{D-CTT18}, see Section~\ref{s:crossing}.

At this point, problems concerning the random walk can be written in terms of box-crossing estimates.
This motivates us to prove Russo-Seymour-Welsh (RSW) estimates (Propositions~\ref{l:v_rsw} and \ref{p:h_rsw}), which in turn provide us with the desired bound on the fluctuations of the random walk.

Let us finally explain where in the proof we require the random walk to be weakly coupled with the environment (see the constant $\delta$ in Theorem~\ref{t:main}).
The decoupling that we require for the underlying random environment does not work for two sets that are very long and near one another.
But during our proof, it becomes necessary to decouple two boxes of height $t$ that are separated by (roughly speaking) $\sqrt{\Var(X_t)}$.
This means that, in order to prove a lower bound on the fluctuations at time $2t$, we need as an input a lower bound on the fluctuations at time $t$.
By using a perturbative argument, we can compare the random walk in random environment to a simple random walk, obtaining a lower bound on its fluctuations at the initial scale and bootstrap the estimate from there.

\paragraph{Open Problems.}

While writing this article we came across a series of interesting open questions, some of which we list below.
\begin{enumerate}[\quad a)]
\item Is it possible to remove the perturbative assumption $\delta \leq \uc{c:perturbative}$ in Theorem~\ref{t:main}?
\item What about dropping other assumptions, for example concerning plannarity, symmetry or FKG inequality?
\item Concerning upper bounds on the fluctuations of $X_n$, is it true that there exists $\xi > 0$ such that $\lim_{n} \mathbb{P}\big( |X_{n}| \geq n^{1 - \xi} \big) = 0$?
\item With a strong mixing assumption on the environment (say large $\uc{c:alpha}$ or small $\varepsilon$ in Definition~\ref{def:decoupling}), is one able to show diffusiveness and possibly a CLT for $X_n$?
\item Is there an example of random environment that fits the main assumptions of our paper, but for which one expects to observe anomalous diffusions (sub or super-diffusive)?
\end{enumerate}

\begin{remark}
\label{r:constants}
Throughout the text, numbered constants like $c_0, c_1,\ldots$ will have their values fixed at their first appearance and will remain unchanged from that point on.
\end{remark}

\paragraph{Organization.}

The rest of the paper is organized as follows.
The graphical construction of the random walk is presented in Section~\ref{notation}.
Section~\ref{s:crossing} contains the introduction of the main class of events used in the proof of Theorem~\ref{t:main}, namely, the crossing events.
In Section~\ref{s:rsw} our main technical result, a Russo-Seymour-Welsh type of estimate, is proved.
The proof of Theorem~\ref{t:main} is then concluded in Section~\ref{s:bootstrap}.
Finally, we verify that the environments given by the Gaussian field and the Confetti percolation satisfy the hypotheses of Theorem~\ref{t:main} in Section~\ref{s:examples}, establishing Corollaries~\ref{t:gf} and~\ref{t:confetti}.

\paragraph{Acknowledgements.}

RB has counted on the support of the Mathematical Institute of Leiden University and ``Conselho Nacional de Desenvolvimento Científico e Tecnológico – CNPq'' grant ``Produtividade em Pesquisa'' (308018/2022-2).
The research of MH was partially supported by CNPq grants ``Projeto Universal'' (406659/2016-8) and ``Produtividade em Pesquisa'' (312227/2020-5) and by FAPEMIG grant ``Projeto Universal'' (APQ-01214-21).
The research of DK was partially supported by the EPSRC grant EP/V00929X/1.
During this period, AT has also been supported by grants ``Projeto Universal'' (406250/2016-2) and ``Produtividade em Pesquisa'' (304437/2018-2) from CNPq and ``Jovem Cientista do Nosso Estado'' (202.716/2018) from FAPERJ.

\section{Graphical construction of the random walk}
\label{notation}
~
\par Here we use a graphical construction in order to define a family of coupled continuous space-time random walks $(X^u_t, u \in \mathbb{R}^2, t \ge 0)$.
Here $u$ represents the starting point of the walk on the plane, while $t$ controls its continuous time evolution.

\nc{c:ellipticity}
We also impose an extra condition that acts as a form of \emph{uniform ellipticity} for the random walk.
Suppose that.
\begin{equation}
  \label{e:ellipticity}
  \delta < \frac{1}{2} - \uc{c:ellipticity},  \text{ for some $\uc{c:ellipticity} \in \Big( 0,\frac12 \Big)$},
\end{equation}
which readly implies that
\begin{equation}
  \label{e:unif_ellipticity_0}
  \mathbb{P}^{f} \big( X_1 \geq X_0 + 1 \big) \geq \uc{c:ellipticity}, \quad \text{uniformly over $f$.}
\end{equation}
We will also write $\mathbb{P}$ for the annealed law of the random walk $(X_n)_{n \geq 0}$.

We now informally state the properties of this coupling that will be useful later on.
Consider the canonical projections $\pi_1, \pi_2: \mathbb{R}^2 \to \mathbb{R}$. Each random walk $X^u := (X^u_t)_{t \geq 0}$ will be such that, $X^u_0 = \pi_1(u)$ almost surely.
Moreover, $X^{u'}$ and $X^u$ coalesce if they ever intersect, that is if $X^{u'}_{t'} = X^u_t$  for some $u, u' \in \mathbb{R}^2$ and $t, t' \geq 0$, then $X^{u'}_{t' + s} = X^u_{t + s}$ for all $s \geq 0$.
Finally, $(X^{(0, 0)}_n)_{n \in \mathbb{N}}$ will have the same law as $(X_n)_{n \in \mathbb{N}}$ defined in~\eqref{e:X_n}.
Knowing what we are attempting to define, we can jump into the rigorous definition of the process.
This graphical construction is inspired by the one appearing in \cite{HKT19}.

As $(X_n)_{n \in \mathbb{N}}$ is assumed to be a nearest-neighbor random walk, $X_{2n} \in 2 \mathbb{Z}$ and $X_{2n + 1} \in \left( 1 + 2\mathbb{Z} \right)$, for every $n \geq 0$.
This motivates us to define the discrete lattice
\begin{equation}
\label{e:def_space_time}
\mathbb{L}_d:=\left( 2\mathbb{Z} \right)^2\cup  \left( (1,1)+\big(2\mathbb{Z}\right)^2 \big),
\end{equation}
where the sum in the right-hand side stands for the the shift of the $2\mathbb{Z}$ lattice by the vector $(1,1)$.
We will first define the random walks $X^u$ for $u \in \mathbb{L}_d$ and later we will interpolate this definition appropriately for any $u \in \mathbb{R}^2$. 
Recall that $f: \mathbb{Z}^2 \to \mathbb{R}$ is a random environment sampled according to $\P$.
Let $\left(U_u\right)_{u\in\mathbb{L}_d}$ be a collection of i.i.d.\ uniform random variables on $[0,1]$.
For any  $u=(x,n)\in \mathbb{L}_d$, we set $X^u_0=x$ and define $X^u_1$ in the following manner:
\begin{equation}
  \label{e:def_X}
  \begin{split}
    X^u_1 & =
    \begin{cases}
      x+2\, \mathbf{1}_{\{U_u \ge 1/2\}} - 1, & \text{ if $f(x, n) = 0$};\\
      x+2\, \mathbf{1}_{\{U_u \ge 1/2 - \delta\}} - 1, & \text{ if $f(x, n) = 1$};\\
      x+2\, \mathbf{1}_{\{U_u \ge 1/2 + \delta\}} - 1, & \text{ if $f(x, n) = -1$};\\
    \end{cases}\\
    & = x + 2 \mathbf{1}_{\{ U_u \geq 1/2 - \delta f(x, n) \}} - 1.
  \end{split}
\end{equation}

Define now inductively for $m\ge1$
\begin{equation}
X^u_m=X^u_{m-1}+X^{(X^u_{m-1},\pi_2(u)+m-1)}_1.
\end{equation}
This defines the coupled family $(X^u_n,u\in\mathbb{L}_d,n\in\mathbb{N})$.
Note that $(X^u_n,\pi_2(u)+n)_{n\in \mathbb{N}}$ evolves on $\mathbb{L}_d$.

In order to extend the random walkers to all positive times, we will join the points of $\mathbb{L}_d$ with continuous edges, through the definition
\begin{equation}
\mathbb{L} = \Big\{ u + t(1, 1); u \in \mathbb{L}_d, 0 \le t < 1 \Big\} \cup \Big\{ u + t(-1, 1); u\in\mathbb{L}_d, 0 \le t < 1 \Big\},
\end{equation}
see Figure~\ref{f:L_d}.

For $t\in\mathbb{R}^+$ and $u\in\mathbb{L}_d$, define
\begin{equation}\label{badcompany}
X^u_t=X^u_{\lfloor t\rfloor}+\left(t-\lfloor t\rfloor\right)\big(X^u_{\lfloor t\rfloor+1}-X^u_{\lfloor t\rfloor}\big),
\end{equation}
which defines the coupled family $(X^u_t,u\in\mathbb{L}_d,t\in\mathbb{R}_+)$.
Note that, for each $u \in \mathbb{L}_{d}$, $(X^u_t,\pi_2(u)+t)_{t\in \mathbb{R}_+}$ evolves on $\mathbb{L}$.

\begin{figure}[ht] % to reproduce the lipschitz_plus.pdf
  \centering
  \begin{tikzpicture}[scale=0.35]
    \draw[step=10ex,color=gray!20,very thin] (-49.9ex,0.1ex) grid (49.9ex,99.9ex);
    \foreach \x in {-2,...,2}{
      \draw[color=blue!80!black, thin] (-20ex*\x,0ex) -- (-50ex,50ex-20ex*\x) -- (20ex*\x,100ex) -- (50ex,50ex+20ex*\x) -- cycle;
      \draw[thin, color=gray] (-20ex*\x,0ex) circle (0.7ex);
    }
    \node[below] at (0,0){$(0,0)$};
    \filldraw[color=red!80!black, thick] (0,0) circle (0.7ex) -- (-10ex,10ex) circle (0.7ex) -- (0,20ex) circle (0.7ex) -- (-10ex,30ex) circle (0.7ex) -- (0,40ex) circle (0.7ex) -- (-10ex,50ex) circle (0.7ex) -- (-20ex,60ex) circle (0.7ex) -- (-30ex,70ex) circle (0.7ex) -- (-20ex,80ex) circle (0.7ex) -- (-30ex,90ex) circle (0.7ex) -- (-40ex,100ex) circle (0.7ex);
    \filldraw[color=red!80!black, thick] (-12ex,3.5ex) circle (0.4ex) -- (-12ex,8ex) circle (0.4ex) -- (-10ex,10ex);
    \filldraw[color=red!80!black, thick] (8ex,3.5ex) circle (0.4ex) -- (8ex,8ex) circle (0.4ex) -- (10ex,10ex) circle (0.7ex) -- (0,20ex);
    \filldraw[color=red!80!black, thick] (-18ex,25ex) circle (0.4ex) -- (-18ex,38ex) circle (0.4ex) -- (-20ex,40ex) circle (0.7ex) -- (-10ex,50ex);
  \end{tikzpicture}
  \caption{In blue we represent the lattice $\mathbb{L}$, observing that its intersection points are the vertices of $\mathbb{L}_d$.
    Random walk trajectories are depicted in red and one can observe their collision behavior.}
  \label{f:L_d}
\end{figure}
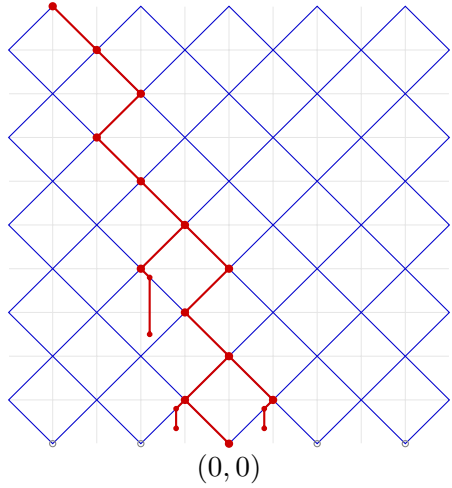

From a starting point $u\in\mathbb{L}\setminus\mathbb{L}_d$, intuitively speaking, we let $X_t^u$ follow the only path such that $(X_t^u, \pi_2(u)+t)$ remains on $\mathbb{L}$ until it hits $\mathbb{L}_d$, after which it follows the rule given by \eqref{badcompany}.
More precisely, given $u \in \mathbb{L} \setminus \mathbb{L}_d$, for any $s>0$, note that $u+s\big((-1)^{k},1\big)\in\mathbb{L}$, where $k=k(u)=\lfloor \pi_1(u)\rfloor+\lfloor \pi_2(u)\rfloor$, and set
\begin{equation}
t_0=\min\big\{s\ge 0: u+s((-1)^{k},1)\in\mathbb{L}_d\big\}.
\end{equation}
We then define
\begin{equation}
\label{horriblecompany}
X^u_t =
\begin{cases}
\pi_1(u)+(-1)^{k}t  & \text{ if $0\le t \leq t_0$};\\
X^{(X^u_{t_0},\pi_2(u)+t_0)}_{t-t_0} & \text{ if $t > t_0$}.
\end{cases}
\end{equation}

Finally, it remains to construct the random walks starting from points $u=(x,t)\in\mathbb{R}^2 \setminus \mathbb{L}$.
The idea is very simple: its trajectory is such that $(X^u_t,\pi_2(u)+t)$ moves up along the direction $(0,1)$ until hitting $\mathbb{L}$ and, from this point on, follows the corresponding trajectory in $\mathbb{L}$ as defined in \eqref{horriblecompany}.

More precisely, given $u=(x,t)\in\mathbb{R}^2$, let
\begin{equation}
s_0=\min\{s\ge 0: (x,t+s)\in\mathbb{L}\}
\end{equation}
and define
\begin{equation}
\label{awfulcompany}
X^u_s =
\begin{cases}
x &  \text{ if $0\le s < s_0$};\\
X^{(x,\pi_2(u)+s_0)}_{s-s_0} & \text{ if $s\ge s_0$}.
\end{cases}
\end{equation}

Equations \eqref{badcompany}, \eqref{horriblecompany}, and \eqref{awfulcompany} define the coupled family $(X^u_t,u\in\mathbb{R}^2,t\in \mathbb{R}_+)$, such that the points $(X^u_t, \pi_2(u)+t)$ always remain on $\mathbb{L}$, after the first time they hit~$\mathbb{L}$.

We also use that, for any $u\in\mathbb{R}^2$ and $t \ge s\ge 0$,
\begin{equation}
  \label{e:lipschitz}
  \left|X^u_t - X^u_s\right| \le t - s.
\end{equation}
Let us now introduce a procedure that will allow us to reflect the the random walk over vertical lines.
Given $a \in \R$ and a random walk trajectory $X^u$ starting to the left of $a$ (i.e. $\pi_{1}(u) \leq a$), we define the random walk reflected at $a$ by forbidding $X^u$ to cross the vertical line $\{x=a\}$ from left to right.
This is done in as follows:
The reflected random walk follows exactly the same rules for the evolution of $X^u$, except that whenever it touches the vertical line $\{x=a\}$ coming along a line segment of the lattice $\mathbb{L}$ oriented in the northeast direction, instead of simply following that segment (as $X$ would do), the new reflected random walk simply travels straight upwards along $\{x=a\}$ until hitting a line segment of $\mathbb{L}$ oriented on the northwest direction.
From that point on, it obeys to the same rules introduced above for $X$ until the next time it will hit the vertical line $\{x=a\}$.
Furthermore, in case the line $\{x=a\}$ contains points of $\mathbb{L}_{d}$ and the random walk $X^u$ should take a jump to the right at any of these points, this jump is suppressed and the reflected random walk travels upwards instead, until the first node of $\mathbb{L}_{d}$ that sends the random walk to the left of the line $\{x=a\}$.
Analogous definitions hold for right reflections.
See Figure~\ref{fig:bouncing} for a representation of these trajectories.

\begin{figure}[ht]
  \centering
  \begin{subfigure}{.45\textwidth}
  \begin{tikzpicture}[scale=0.35]
    \draw[step=10ex,color=gray!20,very thin] (-49.9ex,0.1ex) grid (49.9ex,99.9ex);
    \foreach \x in {-2,...,2}{
      \draw[color=blue!80!black, thin] (-20ex*\x,0ex) -- (-50ex,50ex-20ex*\x) -- (20ex*\x,100ex) -- (50ex,50ex+20ex*\x) -- cycle;
      \draw[thin, color=gray] (-20ex*\x,0ex) circle (0.7ex);
    }
    \draw[thick, black](17ex, 0) -- (17ex, 100ex);

    \filldraw[color=red!80!black, thick] (8ex,3.5ex) circle (0.4ex) -- (8ex,8ex) circle (0.7ex) -- (10ex,10ex) circle (0.7ex) -- (17ex,17ex) circle (0.7ex) -- (17ex,23ex) circle (0.7ex)  -- (10ex, 30ex) circle (0.7ex) -- (0,40ex) circle (0.7ex) -- (10ex,50ex) circle (0.7ex) -- (17ex,57ex) circle (0.7ex) -- (17ex,63ex) circle (0.7ex) -- (10ex,70ex) circle (0.7ex) -- (0,80ex) circle (0.7ex) -- (-10ex,90ex) circle (0.7ex) -- (0,100ex) circle (0.7ex);
  \end{tikzpicture}
  \end{subfigure}
  \begin{subfigure}{.45\textwidth}
  \begin{tikzpicture}[scale=0.35]
    \draw[step=10ex,color=gray!20,very thin] (-49.9ex,0.1ex) grid (49.9ex,99.9ex);
    \foreach \x in {-2,...,2}{
      \draw[color=blue!80!black, thin] (-20ex*\x,0ex) -- (-50ex,50ex-20ex*\x) -- (20ex*\x,100ex) -- (50ex,50ex+20ex*\x) -- cycle;
      \draw[thin, color=gray] (-20ex*\x,0ex) circle (0.7ex);
    }
    \draw[thick, black](20ex, 0) -- (20ex, 100ex);

    \filldraw[color=red!80!black, thick] (8ex,3.5ex) circle (0.4ex) -- (8ex,8ex) circle (0.7ex) -- (10ex,10ex) circle (0.7ex) -- (20ex,20ex) circle (0.7ex) -- (10ex,30ex) circle (0.7ex)  -- (20ex, 40ex) circle (0.7ex) -- (20ex,60ex) circle (0.7ex) -- (10ex,70ex) circle (0.7ex) -- (0,80ex) circle (0.7ex) -- (-10ex,90ex) circle (0.7ex) -- (0,100ex) circle (0.7ex);
  \end{tikzpicture}
  \end{subfigure}
  \caption{The graphical construction of the random walk with left reflection at the black vertical line. In the first image, the vertical line does not intersect $\mathbb{L}_{d}$. The second line contains points with integer coordinates and the reflection behaves slightly different in these points.}
  \label{fig:bouncing}
\end{figure}
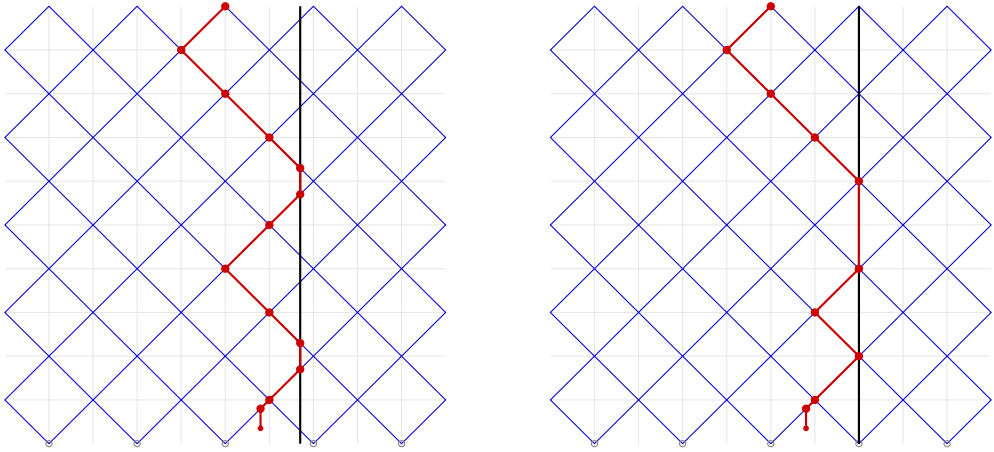

\begin{remark}\label{r:shift}
  It should be noted that the law of $(X_t^u)_{t\ge0}$, for $u\in\mathbb{R}^2$, is not invariant under shifts of $u$.
  Nevertheless, these laws satisfy the following invariance:
  \begin{equation}
    \label{e:invariance}
    (X^u_t)_{t \geq 0} \overset{d}\sim (X^{u + v}_{t})_{t \geq 0}, \text{ for any $v \in \mathbb{L}_{d}$}.
  \end{equation}
  Therefore, the law of all members of the collection $\{(X_t^u)_{t\ge0} \colon u\in\mathbb{R}^2\}$, is fully determined by the law of $\{(X_t^u)_{t\geq 0} \colon u \in \mathcal{L}_1\}$, where
  \begin{equation}\label{e:losange}
    \mathcal{L}_1:=\{u \in\mathbb{R}^2:\ ||u||_{1} \le 1\}.
  \end{equation}
\end{remark}

It is important to make a few comments on the probability space on which the above construction was made.
Denote the uniform distribution on $[0,1]$ by $U[0,1]$ and let $\mathbb{P} := \P \otimes \Pi_{u \in \mathbb{L}_{d}} \U[0,1]$.
This will be the annealed measure of the environment and family of random walks.
Notice that this probability space comprises all the randomness necessary to construct our process.

In analogy to \eqref{e:symmetry}, we now observe an important symmetry property of our random walk trajectory.
For $x_o \in \mathbb{Z}$, we define the reflection along the vertical line $\{x = x_o\}$.
More precisely, let us write $\Omega$ for a sample space supporting both the environment (denoted by $f$) and the random walk (denoted by $g$).
If we let $\bar{\sigma}_{x_o}: \Omega \to \Omega$ be defined by $\bar{\sigma}_{x_o}(f, g) = (f \circ \sigma_{x_o}, g \circ \sigma_{x_o})$, where $\sigma_{x_{o}}(x, y) = (2x_o - x, y)$, the random walk then inherits the symmetry from the environment
\begin{equation}
  \label{e:symmetry_decorated}
  \mathbb{P} \circ \bar{\sigma}_{x_o}^{-1} = \mathbb{P}, \text{ for every $x_o \in \mathbb{Z}$}.
\end{equation}
Observe that the reflections need to be centered around an integer point to guarantee the desired symmetry.

Another important property we need our random walks to satisfy is the monotonicity stated below.

\begin{proposition}[\cite{HKT19}, Proposition~3.1]
  \label{p:monotone}
  For every $\delta \in [0, 1/2]$, for every $u, u' \in \mathbb{R}^2$ with $\pi_1(u')\le \pi_1(u)$ and $\pi_2(u)= \pi_2(u')$, we have that, almost surely,
  \begin{equation}\label{e:monotone1}
    X^{u'}_{t}\le X^{u}_{t},  \text{ for all } t \ge 0.
  \end{equation}
\end{proposition}

Having the notation in place for the random walk starting at any point in $\mathbb{R}^2$, we can re-state our uniform ellipticity assumption in the following form:
\begin{equation}
  \label{e:unif_ellipticity}
  \inf_{f} \inf_{u \in \mathbb{R}^2} \mathbb{P}^{f} \big( X^u_3 > X^u_0 + 1 \big) \geq \uc{c:ellipticity}^3,
\end{equation}
which is a direct consequence of \eqref{e:ellipticity} in the continuous setting.

\begin{remark}
  \label{r:lln}
  Suppose that our environment satisfies the decoupling condition in Defition~\ref{def:decoupling} with $\varepsilon(w, h, r) =  c \big( wh + (w + h) r + r^2 \big) r^{-\alpha}$.
  This means that if $B_1$ and $B_2$ are square boxes with side length $5r$ and mutual distance at least $r$, then any pair of events $A_1$ and $A_2$ that only depend on the environment on these respective boxes satisfy
  \begin{equation}
    \label{e:cov_A1_A2}
    \Cov(A_1, A_2) \leq \varepsilon(5r, 5r, r) \leq c' r^{-\alpha+2}.
  \end{equation}
  Therefore, as soon as $\alpha > 10$ we can use Theorem~2.4 of \cite{10.1214/19-AOP1414} and the symmetry of the walk to conclude that $\lim_{t \to \infty} X_t/t = 0$, almost surely.
\end{remark}

\begin{remark}
  \label{r:recurrent}
  Let us now give a brief argument to show that $X_t$ is almost surely recurrent.
  For this, let $A = \{x \in \mathbb{Z}; X^{(x, 0)}_t \text{ is transient to the right}\}$.
  By~\eqref{e:monotone1}, we conclude that if $x < x'$ and $x \in A$, then $x' \in A$.
  Therefore, there are three possibilities: either $A = \mathbb{Z}$ or $A = \varnothing$ or that there exists $x_0 \in \mathbb{Z}$ such that $A = [x_0, \infty)$.

  Finally, observe that as soon as the error term $\epsilon(w, h, r)$ goes to zero as $r$ diverges (for every $w$ and $h$ fixed), the environment is mixing and therefore ergodic.
  Thus, since the events $[A = \mathbb{Z}]$, $[A = \varnothing]$, and $\cup_{x_0 \in \mathbb{Z}} \big[ A = [x_0, \infty) \big]$ are all tail events, their probability is either zero or one.

  The case $\mathbb{P} [A = \mathbb{Z}] = 1$ is a contradiction with the symmetry of the random walk, while $\mathbb{P} \big[ \cup_{x_0 \in \mathbb{Z}} \big[ A = [x_0, \infty) \big) \big] = 1$ would be a contradiction with the translation invariance of our process, since in this case such special point $x_0$ would be uniformly distributed on $\mathbb{Z}$.
  Therefore we conclude that $\mathbb{P}[A = \varnothing] = 1$, implying that almost surely $X^{(0, 0)}_t$ is not transient to the right (and by symmetry it is also not transient to the left). Thus $X^{(0, 0)}_t$ is almost surely recurrent as claimed.
\end{remark}

For an event $A$ and $u \in \mathbb{R}^2$ we now define $\theta_u \circ A$, the event $A$ shifted by $u$.
This is defined as follows.
Consider the shifted lattice $\theta_u \circ \mathbb{L} := u+ \mathbb{L}$.
Then we define the event $\theta_u \circ A$ such that
\begin{equation}
\Big( (f(x,\cdot))_{x\in \mathbb{Z}^2}, \prod_{v \in \mathbb{L}^d} U[0,1] \Big) \in \theta_u \circ A
\,\,\,\,\,\,\,\,\, \text{iff} \,\,\,\,\,\,\,\,
\Big( (f({x-u}, \cdot))_{x\in\mathbb{Z}^2}, \prod_{v \in -u+\mathbb{L}^d} U[0,1] \Big) \in A.
\end{equation}

\section{Crossing events}
\label{s:crossing}
~
\par This section is devoted to the introduction of a few crossing events that will play central roles in our arguments.

These crossings will take place inside rectangles of the type
\begin{equation}
\label{e:def_box}
B^u(w, h) := u + [0, w] \times [0, h],
\end{equation}
called the box of width $w$ and height $h$ anchored at $u\in \mathbb{R}^2$.
We write $L^u(w, h) := u + \{0\} \times [0, h]$ and $R^u(w, h) := u + \{w\} \times [0, h]$ for its left and right faces, respectively.
Analogously, we define $D^u(w, h)$ and $T^u(w, h)$ for the bottom and top faces, respectively.

Given $A, B \subseteq C \subseteq \mathbb{R}^2$, we define the event that $C$ is crossed from $A$ to $B$ as being
\begin{equation}
  \label{e:crossing}
  \Big[ A \underset{C}\rightarrow B \Big] := \Big[
  \begin{array}{c}
    \text{there exists $v \in A$ such that}\\
    \text{$X^v_t$ hits $B$ before exiting $C$}
  \end{array}
  \Big].
\end{equation}
The most typical situations will consist of $C$ being a box and $A$, $B$ subsets of its faces.

Given a box $B^u(w, h)$, we define the left-right crossing event
\begin{equation}
  \label{e:h_cross}
  H^u(w, h) = \Big[
    L^u(w, h) \underset{B^u(w, h)}\longrightarrow R^u(w, h)
  \Big].
\end{equation}
see Figure~\ref{f:cross} for an illustration.
Analogously, we define the right-left crossing event $\overleftarrow{H}^u(w, h)$ by just reversing the roles of $L^{u}(w,h)$ and $R^u(w,h)$ in \eqref{e:h_cross}.

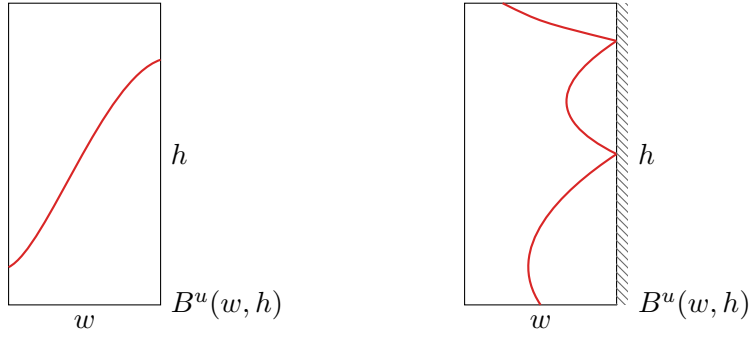
\begin{figure}[h]
  \centering
  \begin{tikzpicture}[scale=.5]
    \draw (0, 0) rectangle (4, 8);
    \draw[thick, color=red!70!gray] (0, 1) .. controls (1, 1.5) and (2.5, 6) .. (4, 6.5);
    \node[right] at (4,0) {$B^u(w, h)$};
    \node[right] at (4,4) {$h$};
    \node[below] at (2,0) {$w$};
    \draw (12, 0) rectangle (16, 8);
    \node[right] at (16.3,0) {$B^u(w, h)$};
    \node[right] at (16.3,4) {$h$};
    \node[below] at (14,0) {$w$};
    \draw[thick, color=red!70!gray] (14, 0) .. controls (13, 1.5) and (14.5, 3) .. (16, 4);
    \draw[thick, color=red!70!gray] (16, 4) .. controls (14, 5) and (14.5, 6) .. (16, 7);
    \draw[thick, color=red!70!gray] (16, 7) .. controls (14, 7.5) .. (13, 8);
    \fill[pattern=north west lines, pattern color=black!60] (16,0) rectangle ++(.3,8);
  \end{tikzpicture}
  \caption{On the left one sees a horizontal crossing $H^u$ from left to right in a box.
    The  picture on the right shows the event $V^u$ where one can see the reflection on the right wall.
  }
  \label{f:cross}
\end{figure}

By the symmetry under reflections (see \eqref{e:symmetry_decorated}) for every $u \in \mathbb{R}^2$, $h>0$ and $w$ such that $\pi_1(u) + w/2 \in \mathbb{Z}$, one has
\begin{equation}
\label{e:symmetry_crossings}
\mathbb{P} (H^u(w, h)) = \mathbb{P} (\overleftarrow{H}^u(w, h)).
\end{equation}
We also define the vertical crossing with reflection
\begin{equation}
  \label{e:v_cross}
  V^u(w, h) :=
  \Big[
    D^u(w, h) \underset{B^u(w, h)|_R}\longrightarrow T^u(w, h)
  \Big],
\end{equation}
where $B^u(w, h)|_R$ stands for the box $B^u(w, h)$ with reflection on its right boundary $u+\{w\} \times [0,h]$, see the paragraph below \eqref{e:lipschitz}. Observe that $V^u(w, h) = \Omega \setminus \overleftarrow{H}^u(w, h)$.

Using \eqref{e:ellipticity} one immediately sees that, as soon as $w \geq 4$,
\begin{equation}
  \label{e:narrow_corridor}
  \inf_{u \in \mathcal{L}_1} \mathbb{P} \big( V^u(w, h) \big) \geq \uc{c:ellipticity}^{h + 2} > 0,
\end{equation}
by simply forcing the random walk to zig-zag its way to the top using uniform ellipticity.

For any $u \in \mathbb{R}^2$ the events $H^u$ and $V^u$ enjoy the following monotonicity properties
\begin{equation}
  \label{e:monotonicity_H}
  \begin{split}
     \text{$\mathbb{P} \big(H^u(w, h)\big)$ is non-increasing in $w$ and non-decreasing in $h$ and}\\
     \text{$\mathbb{P} \big(V^u(w, h)\big)$ is non-decreasing in $w$ and non-increasing in $h$}.
   \end{split}
\end{equation}
  Moreover, if $[a_1, b_1] \subseteq [a_2, b_2]$ and $[c_2, d_2] \subseteq [c_1, d_1]$, then
\begin{equation}
  \label{e:cross_contained}
  \begin{split}
  \text{$H\big([a_2, b_2] \times [c_2, d_2]\big)$ is contained in $H\big([a_1, b_1] \times [c_1, d_1]\big)$;}\\
  \text{$V\big([a_1, b_1] \times [c_1, d_1]\big)$ is contained in $V\big([a_2, b_2] \times [c_2, d_2]\big)$,}
  \end{split}
\end{equation}
as it can be inferred by considering subsets of the original crossing.

As a consequence of \eqref{e:cross_contained} and \eqref{e:invariance}, we have for any $h, w \geq 5$,
\begin{align}
  \label{e:sup_inf_H}
  \sup_{u \in \mathbb{R}^2} \mathbb{P}(H^u(w, h)) & \leq \inf_{u \in \mathbb{R}^2} \mathbb{P}(H^u(w - 4, h + 4)), \\
  \label{e:sup_inf_V}
  \sup_{u \in \mathbb{R}^2} \mathbb{P}(V^u(w, h)) & \leq \inf_{u \in \mathbb{R}^2} \mathbb{P}(V^u(w + 4, h - 4)).
\end{align}
In fact, we can see that for any point $u  \in \mathcal{L}_1$ we have
\begin{equation}
\mathbb{P}\big(H^u(w, h)\big) \leq \mathbb{P}\big(H^0(w-2,h+2)\big) \leq \mathbb{P}\big(H^u(w-4, h+4)\big)
\end{equation}
from where \eqref{e:sup_inf_H} can be deduced.
The argument leading to \eqref{e:sup_inf_V} is similar.

It is intuitive that as we stretch a box vertically, the probability to find a horizontal crossing should increase, since we are given more and more attempts to find such a crossing (and an analogous statement in the other direction).
This intuition can be made precise with help of our decoupling inequality from Definition~\ref{def:decoupling} and it is the content of our next result.

\begin{lemma}
  \label{l:stretch}
  Fix an arbitrary $u \in \mathbb{R}^2$ and any triple $w, h, h_o > 1$ such that $w \leq h$ and $h + h_o$ is an even integer.
  Then
  \begin{equation}
    \label{e:vert_cros_iterated}
    \mathbb{P} \big( V^u \big( w, k(h + h_o) \big) \big)
    \leq \mathbb{P} \big( V^u \big( w, h \big) \big)^k + k \varepsilon(w, h, h_o),
  \end{equation}
  for all $k \geq 1$.
  Moreover, for $w_o > 1$ such that $w + w_o$ is an even integer,
  \begin{equation}
    \label{e:hori_cros_iterated}
    \mathbb{P} \big( H^u \big( k(w + w_o), h \big) \big)
	\leq \mathbb{P} \big( H^u \big( w, h \big) \big)^k + k \varepsilon(w, h, w_o),
  \end{equation}
  for every $k \geq 1$.
\end{lemma}

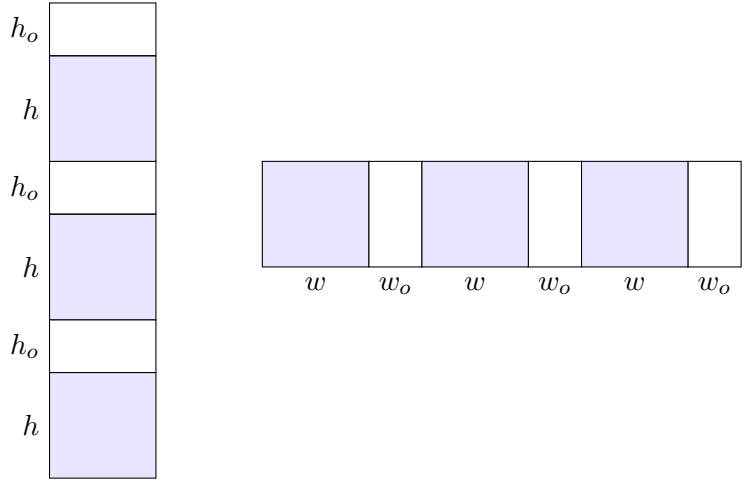
\begin{figure}[h]
  \centering
  \begin{tikzpicture}[scale=.7]
    \foreach \i in {0, 3, 6} {
      \draw[fill=blue!10!white] (0, \i) rectangle (2, \i + 2);
      \draw (0, \i + 2) rectangle (2, \i + 3);
      \node[left] at (0, \i + 1) {$h$};
      \node[left] at (0, \i + 2.5) {$h_o$};
      \draw (\i + 6, 4) rectangle (\i + 7, 6);
      \draw[fill=blue!10!white] (\i + 4, 4) rectangle (\i + 6, 6);
      \node[below] at (\i + 5, 4) {$w$};
      \node[below] at (\i + 6.5, 4) {$w_o$};
    }
  \end{tikzpicture}
  \caption{Splitting elongated boxes vertically and horizontally.}
  \label{f:splitting}
\end{figure}

\begin{proof}
  For the first inequality, define the rectangle $C = u + [0, w] \times [0, k (h + h_o)]$ and apply the decoupling in Definition~\ref{def:decoupling}.

  We start by looking at \eqref{e:vert_cros_iterated}.
  Let us place $k$ parallel horizontal strips of height $h$ vertically separated by distance $h_o$ inside box $B^u(w, k(h+h_o))$, see Figure~\ref{f:splitting}.
  Observe that any vertical crossing of $B^u(w, k(h + h_o))$ has to cross all the $k$ sub-boxes.
  More precisely,
  \begin{equation}
    V^u \big( w, k(h + h_o) \big) \subseteq \bigcap_{j = 0}^{k - 1} V^{u + (0, j(h + h_o))} \big( w, h \big).
  \end{equation}

  We now write $D_j$, with $j = 0, \dots, k - 1$ for these boxes and use Definition~\ref{def:decoupling} with $A = D_{k-1}$ and $B = \cup_{j = 0}^{k - 2} D_j$ to obtain
  \begin{equation}
    \mathbb{P} \big( V^u \big( w, k (h + h_o) \big) \big)
    \leq \mathbb{P}\Big(\bigcap_{j = 0}^{k - 2} V^{u + (0, j(h + h_o))} \big( w, h \big)\Big) \mathbb{P} \big( V^u \big( w, h \big) \big) + \varepsilon(w, h, h_o),
  \end{equation}
where we used that $ \mathbb{P} \big( V^{u+(0,h)}(\cdot,\cdot)\big)=\mathbb{P} \big( V^u(\cdot,\cdot)\big)$ for all $h\in2\mathbb{Z}$.
The proof of \eqref{e:vert_cros_iterated} is finished by iterating the argument $k-1$ times.

  The proof of \eqref{e:hori_cros_iterated} follows the same lines.
One just needs to split the relevant box into vertical strips appropriately.
The details are left for the reader.
\end{proof}

Observe that taking complements in \eqref{e:hori_cros_iterated} and \eqref{e:vert_cros_iterated} give us that
\begin{equation}
  \label{e:stretch_vertical}
  \mathbb{P} \big( V^u \big( k (w + w_o), h \big) \big) \geq 1 - \Big( 1 - \mathbb{P} \big( V^u(w, h) \big) \Big)^k - k \varepsilon(w, h, w_o)
\end{equation}
and
\begin{equation}
  \label{e:stretch_horizontal}
  \mathbb{P} \big( H^u \big( w, k (h + h_o) \big) \big) \geq 1 - \Big( 1 - \mathbb{P} \big( H^u(w, h) \big) \Big)^k - k \varepsilon(w, h, h_o),
\end{equation}
which will be useful later.

We state now some simple consequences of the above bounds on the asymptotic crossing probabilities as the boxes become very long or wide.

\begin{corollary}
  \label{c:trivial_limits}
   For fixed $h > 0$ we have
  \begin{equation}
    \label{e:w_large}
    \lim_{w \to \infty} \inf_{u \in \mathbb{R}^2} \mathbb{P} \big( V^u(w, h) \big) = 1.
  \end{equation}
Moreover, for any $w > 0$ fixed,
  \begin{equation}
    \label{e:h_large}
    \lim_{h \to \infty} \sup_{u \in \mathbb{R}^2} \mathbb{P} \big( V^u(w, h) \big) = 0 \qquad \text{and} \qquad
    \lim_{h \to \infty} \inf_{u \in \mathbb{R}^2} \mathbb{P} \big( H^u(w, h) \big) = 1.
  \end{equation}
\end{corollary}

\begin{proof}
  Recall that $\lim_{r \to \infty} \varepsilon(w, h, r) = 0$ for any fixed $w, h$.
  The limit in \eqref{e:w_large} follows from \eqref{e:stretch_vertical} by fixing $w = 4$ as in \eqref{e:narrow_corridor} and choosing $w_0(k)$ such that $k\varepsilon(4,h,w_0(k))\to0$ as $k$ tends to infinity.
  The limits in \eqref{e:h_large} follow from similar arguments.
\end{proof}

Observe also that
\begin{display}
  \label{e:light_cone}
  if $w > h$, then $\mathbb{P}(H^u(w, h)) = 0$, for any $u \in \mathbb{R}^2$,
\end{display}
by the fact that the walk is $1$-Lipschitz.

\section{Russo-Seymour-Welsh technique}
\label{s:rsw}
~
\par In this section we will prove a result that is inspired by the RSW Theorem for Bernoulli percolation, see \cite{Russo} and \cite{SeymourWelsh}.
In our specific setting, the process is not symmetric with respect to right angle rotations of the plane, therefore the type of statements that we are after will not be stated in terms of square boxes.
Instead, we take inspiration from a similar RSW-type result for oriented percolation, see \cite{D-CTT18}.

In Subsection~\ref{ss:w_h} we introduce an implicit parameter $w(h)$ that will capture the correct scale of fluctuations for a random walk that runs for time $h$.
Intuitively speaking, $w(h)$ will be such that a random walk has probability $1/2$ of fluctuating more than $w(h)$ and probability $1/2$ of fluctuating less during a time interval of length $h$.

These fluctuation probabilities will be stated in terms of the crossing events defined in Section~\ref{s:crossing} and it is the main objective of this section to show that $w(h)$ is lower bounded by a polynomial in $h$.
This is achieved through a RWS-type argument where we show that crossing longer and wider boxes is also likely to happen, allowing us to connect what happens at different scales and bound $w(h)$ from below.

Subsections~\ref{ss:vertical} and \ref{ss:horizontal} are respectively devoted to proving a vertical and a horizontal version of RSW statements, see Propositions~\ref{l:v_rsw} and \ref{p:h_rsw}.

\subsection{The fluctuation scale}
\label{ss:w_h}
~
\par We now introduce the function $w(h)$ that intuitively speaking measures the order of fluctuations that the random walk undergoes after time $h$.
In order to use RSW and percolation techniques, we will define this quantity in terms of crossing events and we will do so by tuning $w(h)$ in a way that horizontal and vertical crossings are equally likely.

Due to the fact that our process is not translation invariant in the full $\mathbb{R}^2$ (but rather only on the lattice $\mathbb{L}_{d}$), we will define this implicit constant in terms of an integral along a rhombus instead.

\begin{definition}
  \label{d:w_h}
  Given $h \geq 2$, we recall the definition of $\mathcal{L}_1$ in \eqref{e:losange} and introduce
  \begin{equation}
    \label{e:w_h}
    w(h) = \sup \Big\{ w \geq 1; \frac{1}{\Vol(\mathcal{L}_1)} \int_{u \in \mathcal{L}_1} \mathbb{P} \big( H^u (w, h) \big) \d u \geq \frac{1}{\Vol(\mathcal{L}_1)} \int_{u \in \mathcal{L}_1} \mathbb{P} \big( V^u (w, h) \big) \d u \Big\},
  \end{equation}
  where the integration is with respect to the Lebesgue measure on the plane.
\end{definition}

\begin{remark}
  \label{r:integral_set}
 By Corollary~\ref{c:trivial_limits} the supremum in \eqref{e:w_h} runs over a set that is bounded above.
 Moreover, note that $\inf_u H^u(1/2, 3) > 0$, while $V^u(1/2, 3) = 0$ for every $u \in \mathbb{R}^2$, so that set is non-empty as soon as $h \geq 2$.
\end{remark}

\begin{remark}\label{r:monotonicity_w}
 The function $h \mapsto w(h)$ is monotone increasing.
 In fact, by \eqref{e:monotonicity_H}, if $h < h'$ then $\mathbb{P}(H^u(w,h)) \leq \mathbb{P}(H^u(w,h'))$ and $\mathbb{P}(V^u(w,h)) \geq \mathbb{P}(V^u(w,h'))$ for every $u \in \mathbb{R}^2$ and $w \geq 1$.
 Therefore,
 \begin{equation}
 \begin{split}
 \Big\{ w \geq 1; \int_{u \in \mathcal{L}_1} & \mathbb{P} \big( H^u (w, h) \big) \d u \geq  \int_{u \in \mathcal{L}_1} \mathbb{P} \big( V^u (w, h) \big) \d u \Big\} \\
  & \subseteq \Big\{ w \geq 1; \int_{u \in \mathcal{L}_1} \mathbb{P} \big( H^u (w, h') \big) \d u \geq  \int_{u \in \mathcal{L}_1} \mathbb{P} \big( V^u (w, h') \big) \d u \Big\}.
 \end{split}
 \end{equation}
 Taking the supremum on both sides yields $w(h') \geq w(h)$.
\end{remark}

We now observe that the integrals appearing in \eqref{e:w_h} are continuous.
More precisely, for $h \geq 2$, the functions
\begin{equation}
  \label{e:crossing_functions}
  d_1: w \mapsto \frac{1}{\Vol(\mathcal{L}_1)}\int_{u \in \mathcal{L}_1} \mathbb{P} \big( H^u(w, h) \big) \d u
  \quad \text{and} \quad
  d_2: w \mapsto \frac{1}{\Vol(\mathcal{L}_1)}\int_{u \in \mathcal{L}_1} \mathbb{P} \big( V^u(w, h) \big) \d u\\
\end{equation}
are continuous in $w$.
To see why this is true, fix $w^* > 0$ and note that
\begin{equation*}
  \Big\{ u \in \mathcal{L}_1; \mathbb{P} \big( H^u(w, h) \big) \text{ is discontinuous at $w = w^*$} \Big\}
  \subseteq \Big\{ u \in \mathcal{L}_1; \pi_1(u) + w^* \in \mathbb{Z} \Big\},
\end{equation*}
which has zero Lebesgue measure, proving the statement for $d_1$.
The argument for $d_2$ is based on the fact that the two functions in \eqref{e:crossing_functions} are linear combinations of one another, or more precisely $d_1 = 1 - d_2$, due to the fact that $V^u(w, h) = \Omega \setminus \overleftarrow{H}^u(w, h)$ and~\eqref{e:symmetry_crossings}.
This, together with their continuity, implies that for every $h \geq 2$,
\begin{equation}
  \label{e:cross_wh_is_half}
  \frac{1}{\Vol(\mathcal{L}_1)} \int_{u \in \mathcal{L}_1} \mathbb{P} \big( H^u (w(h), h) \big) \d u = \frac{1}{\Vol(\mathcal{L}_1)} \int_{u \in \mathcal{L}_1} \mathbb{P} \big( V^u (w(h), h) \big) \d u = \frac{1}{2}.
\end{equation}

Let us now collect some useful facts about the function $w(h)$.
First, it is important to observe that
\begin{display}
  \label{e:w_h_diverges}
  $w(h)$ goes to infinity as $h$ diverges.
\end{display}
In fact, observe that for any fixed $w^* \geq 1$, if $h \geq k (w^* + 4)$ we can use uniform ellipticity to obtain the following bound
\begin{equation}
  \inf_{u \in \mathbb{R}^2} \mathbb{P}(H^u(w^*, h)) \geq 1 - \big( 1 - \uc{c:ellipticity}^{w^* + 4} \big)^k,
\end{equation}
which is strictly larger than $1/2$ for $k_0 = k_0(\uc{c:ellipticity})$ large enough.
Therefore we know that $w(h) \geq w^{*}$ for $h \geq k_0(\uc{c:ellipticity}) (w^* + 4)$, yielding \eqref{e:w_h_diverges}.

\nc{c:large_h}
Next, we would like to be able replace the integral over $\mathcal{L}_1$ with a worst case lower bound.
Since $w(h)$ diverges as $h$ grows, we can assume that $\uc{c:large_h} = \uc{c:large_h}(\uc{c:ellipticity})>0$ is large enough as to allow the use of~\eqref{e:sup_inf_H} in order to obtain
\begin{equation}
  \label{e:horizontal_lower}
  \begin{split}
    \inf_{u \in \mathbb{R}^2} \mathbb{P} \big( H^u(w(h) - 4, h + 4) \big) & \geq \sup_{u \in \mathbb{R}^2} \mathbb{P} \big( H^u(w(h), h) \big) \\
    & \geq
 \frac{1}{\Vol(\mathcal{L}_1)} \int_{u \in \mathcal{L}_1} \mathbb{P} \big( H^u (w(h), h) \big) \d u = \frac{1}{2}
\end{split}
\end{equation}
and
\begin{equation}
  \label{e:vertical_lower}
  \begin{split}
    \inf_{u \in \mathbb{R}^2} \mathbb{P} \big( V^u(w(h) + 4, h - 4) \big) & \geq \sup_{u \in \mathbb{R}^2} \mathbb{P} \big( V^u(w(h), h) \big) \\
    & \geq
 \frac{1}{\Vol(\mathcal{L}_1)} \int_{u \in \mathcal{L}_1} \mathbb{P} \big( V^u (w(h), h) \big) \d u = \frac{1}{2},
\end{split}
\end{equation}
for all $h \geq \uc{c:large_h}(\uc{c:ellipticity})$.

\subsection{Vertical RSW}
\label{ss:vertical}
~
\par As we have observed in the previous section, a rectangle of height $h$ and width $w(h)$ has a good chance of being crossed both horizontally and vertically.
However, this results is not useful if we do not have some wiggle room to perform geometrical constructions.

This issue is intimately related to the fact that, for Bernoulli percolation, crossing squares with positive probability is not nearly as useful as crossing rectangles in the hard direction.

The next proposition states that crossing much higher boxes (roughly with dimensions $w(h) \times 5h$) has also a uniformly positive probability.

\nc{c:vertical_rsw}
\begin{proposition}[Vertical RSW]
  \label{l:v_rsw}
Let $w(h)$ be defined as in \eqref{e:w_h}.
For every $h \geq 64 \vee \uc{c:large_h}(\uc{c:ellipticity})$, we have
  \begin{equation}
    \label{e:v_rsw}
    \inf_{u \in \mathbb{R}^2} \mathbb{P} \big( V^u(w(h) + 4, 5 h) \big) \geq \uc{c:vertical_rsw},
  \end{equation}
  where $\uc{c:vertical_rsw} := \uc{c:ellipticity}^{360}/{16^{30}}$.
\end{proposition}

\begin{proof}
  We split the proof in two cases, depending on the value of $\mathbb{P} (H^0(w(h) - 8, 3h/4))$.

  {\bf Case 1 - }
  We first consider the possibility that
  \begin{equation}
    \label{e:v_rsw_case1}
    \inf_{u \in \mathbb{R}^2} \mathbb{P}\big( H^u(w(h) - 8, 3h/4) \big) \geq 1/8.
  \end{equation}
  In this case, we first observe that, by uniform ellipticity, we can extend slightly the horizontal crossing.
  More precisely
  \vspace{-4mm}
  \begin{equation}
    \label{e:v_rsw_case1_ellip}
    \inf_{u \in \mathbb{R}^2} \mathbb{P} \big( H^u (w(h) + 4, 3h/4 + 12) \big)
    \geq \mathbb{P} \Bigg( \extend \Bigg)
    \geq \frac{\uc{c:ellipticity}^{12}}{8}.
  \end{equation}
  We can then use this horizontal crossing to glue together two vertical crossings using~\eqref{e:fkg} as described below.
  Fix $u \in \mathbb{R}^2$ and observe that since $h \geq 64$, we have $h-4 \geq 3h/4 +12$ therefore,
  \vspace{-2mm}
  \begin{equation*}
    V^u \big( w(h) + 4, 5h \big) \supseteq \bigcap_{i = 0}^{29} V^{u + (0, i h/5)} \big( w(h) + 4, h - 4 \big) \cap H^{u + (0, i h/5)} \big( w(h) + 4, 3h/4 + 12 \big),
  \end{equation*}
  as depicted in Figure~\ref{f:glue_vertical}.

  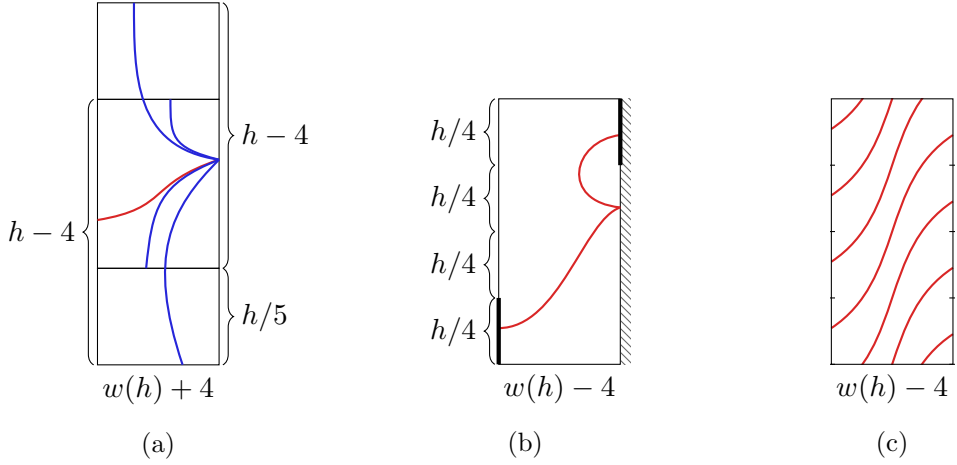
\begin{figure}[h]
    \begin{subfigure}{.3\textwidth}
      \centering
      \begin{tikzpicture}[scale=.16]
        \draw (0, 0) rectangle ++(10, 8);
        \draw (0, 8) rectangle ++(10, 14);
        \draw (0, 22) rectangle ++(10, 8);
        \draw[thick, color=red!70!gray] (0, 12) .. controls (6, 13) and (4, 15) .. (10, 17);
        \draw[thick, color=blue!70!gray] (7, 0) .. controls (4, 9) and (6, 13) .. (10, 17);
        \draw[thick, color=blue!70!gray] (10, 17) .. controls (6, 18) and (6, 19) .. (6, 22);
        \draw[thick, color=blue!70!gray] (4, 8) .. controls (4.5, 13) and (5, 15) .. (10, 17);
        \draw[thick, color=blue!70!gray] (10, 17) .. controls (3, 18) and (3, 24) .. (3, 30);
        \node[below] at (5, 0) {$w(h) + 4$};
        \draw [decorate,decoration={brace,amplitude=4,raise=.3ex},yshift=0pt] (10, 8) -- (10, 0);
        \node[right] at (11, 4) {$h/5$};
        \draw [decorate,decoration={brace,amplitude=4,raise=.3ex},yshift=0pt] (0, 0) -- (0, 22);
        \node[left] at (-1, 11) {$h - 4$};
        \draw [decorate,decoration={brace,amplitude=4,raise=.3ex},yshift=0pt] (10, 30) -- (10, 8);
        \node[right] at (11, 19) {$h - 4$};
      \end{tikzpicture}
      \caption{}
    \end{subfigure}
    \begin{subfigure}{.3\textwidth}
      \centering
      \begin{tikzpicture}[scale=.16]
      \fill[pattern=north west lines, pattern color=black!60] (10,0) rectangle ++(0.8, 22);
        \fill[color=white] (0, 0) rectangle ++(10, 30);
        \draw (0, 0) rectangle ++(10, 22);
        \node[below] at (5, 0) {$w(h) - 4$};
        \draw[thick, color=red!70!gray] (0, 3) .. controls (5, 3) and (7, 12) .. (10, 13);
        \draw[thick, color=red!70!gray] (10, 13) .. controls (5, 13.5) and (6, 18.5) .. (10, 19);
        \draw[ultra thick] (0, 0) -- (0, 5.5);
        \draw[ultra thick] (10, 16.5) -- (10, 22);
        \foreach \i in {0, 5.5, 11, 16.5} {
          \draw [decorate,decoration={brace,amplitude=4,raise=.3ex},yshift=0pt] (0, \i) -- (0, \i + 5.5);
          \node[left] at (-1, \i + 2.7) {$h/4$};
        }
      \end{tikzpicture}
      \caption{}
    \end{subfigure}
    \begin{subfigure}{.3\textwidth}
      \centering
      \begin{tikzpicture}[scale=.16]
        \fill[color=white] (0, 0) rectangle ++(10, 30);
        \node[below] at (5, 0) {$w(h) - 4$};
        \begin{scope}
          \clip (-.1, -.03) rectangle (10.1, 22);
          \foreach \i in {-16.5, -11, -5.5, 0, 5.5, 11, 16.5} {
            \draw (-.3, \i) -- (.3, \i) (10 - .3, \i) -- (10.3, \i);
            \draw[thick, color=red!70!gray] (0, \i + 3) .. controls (6, \i + 7) and (4, \i + 15) .. (10, \i + 19);
          }
          \draw (0, 0) rectangle (10, 22);
        \end{scope}
      \end{tikzpicture}
      \caption{}
    \end{subfigure}
    \caption{Vertical crossings are depicted in blue, while horizontal ones are in red.
    (a) Vertical crossings in overlapping boxes can be joined together in the presence of a short horizontal crossing inside the overlap region, as used in the first case of the proof of Proposition~\ref{l:v_rsw}.
      (b) The crossing from \eqref{e:s-shaped}. (c) The crossing strategy leading to \eqref{eq:idontknow}.
      }
    \label{f:glue_vertical}
    \end{figure}
The above inclusion implies by~\eqref{e:fkg} that
  \begin{equation*}
    \begin{array}{e}
      \mathbb{P} \big( V^u (w(h) + 4, 5h) \big) & \geq & \Big( \inf_{u' \in \mathbb{R}^2} \mathbb{P} \big( V^{u'}(w(h) + 4, h - 4) \big) \Big)^{30}\\
      & & \times \Big( \inf_{u' \in \mathbb{R}^2} \mathbb{P} \big( H^{u'}(w(h) + 4, 3h/4 + 12) \big) \Big)^{30} \;\; \overset{\eqref{e:vertical_lower}, \eqref{e:v_rsw_case1_ellip}}\geq \;\; \frac{\uc{c:ellipticity}^{360}}{16^{30}} = \uc{c:vertical_rsw},
  \end{array}
  \end{equation*}
  yielding \eqref{e:v_rsw}.

  {\bf Case 2 - } We now consider the possibility that
  \begin{equation}
    \inf_{u \in \mathbb{R}^2} \mathbb{P}\big( H^u(w(h) - 8, 3h/4) \big) < 1/8.
  \end{equation}
  In this case, we can slightly widen the box and use \eqref{e:cross_contained} to obtain
  \begin{equation}
    \sup_{u \in \mathbb{R}^2} \mathbb{P} \big( H^u(w(h) - 4, 3h/4 - 4) \big) < 1/8.
  \end{equation}

  The bound above, together with \eqref{e:horizontal_lower} yields
  \begin{equation}
    \label{e:s-shaped}
    \inf_{u \in \mathbb{R}^2} \mathbb{P} \Big( H^u (w(h) - 4, h + 4) \setminus
    \big(
    H^u (w(h) - 4, 3h/4) \cup H^{u + (0, h/4)} (w(h) - 4, 3h/4)
    \big)
    \Big) \geq 1/4.
  \end{equation}
  Notice that the event in the probability above is contained in the increasing event visualized in Figure~\ref{f:glue_vertical} (b), that can be written more concisely as
  \begin{equation}
    \label{e:bound_s}
    \inf_{u \in \mathbb{R}^2} \mathbb{P} \Big(
    \{0\} \times [0, h/4] \underset{B^u(w, h)|_R}\longrightarrow \{w(h) - 4\} \times [3h/4, h]
    \Big) \geq 1/4.
  \end{equation}

  We now observe the following inclusion, depicted in Figure~\ref{f:glue_vertical} (c)
  \begin{equation*}
    V^u \big( w(h) + 4, 5h \big) \supseteq \bigcap_{i = -1}^{28} \Big[
    \{0\} \times \big[ \tfrac{ih}{4}, \tfrac{(i + 1)h}{4} \big] \underset{B^u(w, h)|_R}\longrightarrow \{w(h) - 4\} \times \big[ \tfrac{(i + 3)h}{4}, \tfrac{(i + 4)h}{4} \big]
    \Big],
  \end{equation*}
  which together with \eqref{e:bound_s} and~\eqref{e:fkg} gives
  \begin{equation}\label{eq:idontknow}
    \inf_{u \in \mathbb{R}^2} \mathbb{P} \big( V^u \big( w(h) + 4, 5h \big) \big) \geq \Big( \frac{1}{4} \Big)^{30} \geq \uc{c:vertical_rsw},
  \end{equation}
  as desired.
\end{proof}

Similarly to other RSW-type results, we can stretch the box even further if necessary.

\begin{corollary}
  \label{c:stretch_vertically}
  For every $h \geq 64 \vee \uc{c:large_h}(\uc{c:ellipticity})$
  \begin{equation}
    \label{e:stretch_vertically}
    \inf_{u \in \mathbb{R}^2} \mathbb{P} \big( V^u(w(h) + 4, j h) \big) \geq \Big( \frac{\uc{c:vertical_rsw}\uc{c:ellipticity}^{8}}{2}\Big)^j,
  \end{equation}
  for every $j \geq 1$.
\end{corollary}

\begin{proof}
Define $u_{k} = u+k(0, 4h-12)$, for $k \leq j$, and notice that
\begin{equation}
V^u(w(h) + 4, j h) \supseteq \bigcap_{k=0}^{j} V^{u_{k}}(w(h) + 4, 5h) \cap H^{u_{k+1}}(w(h)+4, h+12).
\end{equation}

The claim now follows from \eqref{e:horizontal_lower}, the \eqref{e:fkg} inequality and the fact that
\begin{equation}
\inf_{u \in \mathbb{R}^2} \mathbb{P} \big( H(w(h)+4, h+12) \big) \geq \frac{\uc{c:ellipticity}^{8}}{2},
\end{equation}
which can be deduced using the ellipticity assumption as in~\eqref{e:v_rsw_case1_ellip}.
\end{proof}

\subsection{Horizontal RSW}
\label{ss:horizontal}
~
\par In analogy with the vertical RSW stated in Proposition~\ref{l:v_rsw}, we now prove a horizontal version of this result.
But as observed in \cite{D-CTT18} for oriented percolation, there is no hope to keep extending horizontal crossings indefinitely, see \eqref{e:light_cone}.
Instead, what we prove is that it is possible to cross a wider box, as long as we stretch the vertical direction as well.

There is another important difference between the vertical RSW result above and its horizontal counterpart: in Proposition~\ref{p:h_rsw} we need to impose a lower bound on $w(h)$ as an extra hypothesis.
This restriction is the main reason why we can only deal with perturbative random walks on this paper.
See Remark~\ref{r:h_rsw_hyp} below for a more in-depth discussion on this.

\begin{proposition}
  \label{p:h_rsw}
  Fix $\xi \in (0,1)$ and $\uc{c:alpha}>1+\xi^{-1}$.
  There exists $\uc{c:h_large_rsw} = \uc{c:h_large_rsw}(\uc{c:ellipticity}, \uc{c:bad_coupling}, \uc{c:alpha}, \xi) > 0$ such that, if $h \geq \uc{c:h_large_rsw}$ and
  \begin{equation}
    w(h) \geq h^{\xi},
  \end{equation}
  then
  \begin{equation}
    \inf_{u \in \mathbb{R}^2}\mathbb{P} \Big( H^u \big( \tfrac{259}{256} w(h), 3h \big) \Big) \geq \frac{\uc{c:ellipticity}^{54}}{2^{10}}.
  \end{equation}
\end{proposition}

Proposition \ref{p:h_rsw} is a straightforward consequence of Lemma \ref{l:end_too_left} and Lemma \ref{l:hard_horizontal_rsw} stated and proved below. 
Similarly to what was done in Corollary~\ref{c:stretch_vertically}, after we have established the above proposition, we can stretch the box horizontally even further.

\nc{c:c_rsw}
\begin{corollary}
  \label{c:stretch_horizontally}
  Fix $\xi \in (0,1)$ and $\uc{c:alpha} > 1+\xi^{-1}$.
  There exists $\uc{c:c_rsw} = \uc{c:c_rsw}(\uc{c:ellipticity}) > 0$ such that, if $h \geq \uc{c:h_large_rsw}$ and
  \begin{equation}
    w(h) \geq h^{\xi} \vee 512
  \end{equation}
  then
  \begin{equation}
    \inf_{u \in \mathbb{R}^2}\mathbb{P} \Big( H^u \Big( \tfrac{128+j}{128} w(h), 3jh \Big) \Big) \geq \uc{c:c_rsw}^{j},
  \end{equation}
  for every $j \geq 1$.
\end{corollary}

\begin{proof}
Set $\uc{c:c_rsw} = \Big( \frac{\uc{c:vertical_rsw}^{6}\uc{c:ellipticity}^{102}}{2^{16}}\Big)^{2}$. Let, for $k \leq j$, $u_{k} = u+ \Big( k\big(\tfrac{3}{256} w(h)-4 \big), 3kh \Big)$ and notice that
\begin{equation}
H^u \Big( \tfrac{128+j}{128} w(h), 3jh \Big) \supseteq \bigcap_{k=0}^{j-1} H^{u_{k}} \Big( \tfrac{259}{256} w(h), 3h \Big) \cap V^{u_{k}-(0,3h)} \big( w(h)+4, 6h \big).
\end{equation}
The claim follows now from Proposition~\ref{p:h_rsw}, Corollary~\ref{c:stretch_vertically}, and the FKG inequality.
\end{proof}

Before the proof of the proposition, we need to introduce some extra notation.
By the definition of $w(h)$, we know that crossing a $w(h) \times h$ box vertically has positive probability.
However during this section we need to consider the probability of a vertical crossing when the random walk starts from the middle of the basis of the box.
In what follows we will properly define this modified crossing event that better explores the symmetry of our process.

Fixed $w, h \geq 1$, we define the vertical crossing from the bottom middle point as:
\begin{equation}
  \label{e:NE}
  \dot{V}^u(w, h) := \vdot{w}{h}{u} =
  \theta_u \circ
  \Big[
    \{ (w/2, 0) \} \underset{B(w, h)|_R}\longrightarrow T(w, h)
  \Big],
\end{equation}

Although the event $\dot{V}$ is less likely than $V$, we still have a uniform lower bound for its probability, as stated in the next lemma.

\begin{lemma}
  \label{l:lower_v_dot}
  Let $h \geq \uc{c:large_h}(\uc{c:ellipticity})$ and consider the corresponding $w(h)$ as in Definition~\ref{d:w_h}.
  Then, under the symmetry assumption \eqref{e:symmetry},
  \begin{equation}
    \label{e:lower_v_dot}
    \mathbb{P} \big( \dot{V}^u(w(h) + 4, h - 4) \big) > 1/8,
  \end{equation}
  for every $u \in \mathbb{R}^2$ such that $\pi_1(u) + w(h)/2 \in \mathbb{Z}$.
\end{lemma}

Observe that in the above statement we restrict ourselves to some specific values of~$u$.
This is due to the fact that our proof uses the symmetry of the process.

\begin{proof}
  Start by considering the event where the random walk starting from the center of the box first hits its right face $u + \{w(h)+4\} \times [0, h-4]$ and the event where the random walk first hits the top face. More specifically, consider
  \begin{equation}
    \vbar{w}{h}{u} = \theta_u \circ \bigg[
    \{(w/2, 0)\} \underset{(0, w) \times [0, h]}\longrightarrow
    \{w\} \times (0, h]
    \bigg]
  \end{equation}
and
  \begin{equation}
    \vdotdotgood{w}{h}{u} = \theta_u \circ \bigg[
    \{(w/2, 0)\} \underset{(0, w] \times [0, h]}\longrightarrow
    T(w,h)
    \bigg].
  \end{equation}

Notice that
\begin{equation}
\vdotdotgood{w}{h}{u} \bigcup \bigg(\vbar{w}{h}{u} \cap V^{u}(w,h) \bigg) \subseteq \dot{V}^u(w,h),
\end{equation}
since in the first case the path that realizes the event also provides a crossing verifying $\dot{V}^{u}(x,h)$. In the second event of the union above, the paths that realize each of the events must intersect thus producing a crossing verifying the occurence of $\dot{V}^{u}(x,h)$.

In particular, the lemma follows if
  \begin{equation}
  \label{eq:top_hit_first}
    \mathbb{P} \bigg( \vdotdotgood{w(h) + 4}{h - 4}{u} \bigg) \geq \frac{1}{8}.
  \end{equation}

  Assume now that this is not the case.
  We now observe that the random walk starting from the central point of $B^{u}(w(h) + 4, h - 4)$ first hits the boundary of this box in one of these three sets: the \emph{left face} $u + \{0\} \times [0, h-4]$, the \emph{right face} $u + \{w(h)+4\} \times [0, h-4]$, or the \emph{top face} $u+[0, w(h)+4] \times \{h-4\}$.
  Due to the fact that $\pi_1(u) + w(h)/2$ is assumed to be integer, the probabilities of first hitting the left boundary or the right boundary are equal.
  In particular,
  \begin{equation}
    \label{e:lower_hit_right}
    \mathbb{P} \bigg( \vbar{w(h) + 4}{h - 4}{u} \bigg) = \frac{1}{2}\bigg(1 - \mathbb{P} \bigg( \vdotdotgood{w(h) + 4}{h - 4}{u} \bigg) \bigg) \geq \frac{1}{4},
  \end{equation}
since we are assuming that~\eqref{eq:top_hit_first} does not hold.

This implies, via~\eqref{e:fkg},
\begin{equation}
\mathbb{P} \bigg( \vbar{w(h) + 4}{h - 4}{u} \cap V^{u}(w(h) + 4, h - 4) \bigg) \geq \mathbb{P} \bigg( \vbar{w(h) + 4}{h - 4}{u} \bigg) \mathbb{P} \big( V^{u}(w(h) + 4, h - 4) \big) \geq \frac{1}{8},
\end{equation}
where the last inequality follows by combining~\eqref{e:vertical_lower},~\eqref{e:lower_hit_right}, and the fact that $h \geq \uc{c:large_h}(\uc{c:ellipticity})$. This concludes the proof.
\end{proof}

For the proof of our horizontal RSW result (Proposition~\ref{p:h_rsw}), we need to consider two separate cases.
Intuitively speaking, when we start a random walk from the middle point of the bottom of a box (as in the definition of $\dot{V}$), we are interested in its horizontal position as it reaches the top of that box.
We will split the proof of Proposition~\ref{p:h_rsw} in two cases, depending on whether this random walk ends too much to the left of the box or not.

We first need to introduce the event
\begin{equation}
  \label{e:dot_dot}
  \vdotdot{w}{z}{h}{u} :=
  \theta_u \circ
  \Big[
    \{ (w/2, 0)\} \underset{B(w, h)|_R}\longrightarrow [0, z] \times \{h\}
  \Big]
\end{equation}

In the next lemma, we start to prove Proposition~\ref{p:h_rsw}.
More precisely, we first treat the case when a random walk starting at the middle point of the bottom of the box has a good probability to end up too much to the left.

\begin{lemma}
  \label{l:end_too_left}
  Suppose that for some $h \geq 64 \wedge \uc{c:large_h}(\uc{c:ellipticity})$ and some $u^* \in \mathbb{R}^2$ such that $\pi_1(u^*) + w(h)/2 \in \mathbb{Z}$ we have
  \begin{equation}\label{assemi1}
    \mathbb{P} \Bigg( \vdotdot{w(h) + 4}{w(h)/8}{h + 12}{u^*} \Bigg) > \gamma.
  \end{equation}
  Then
  \begin{equation}
    \label{e:end_to_left}
    \inf_{u \in \mathbb{R}^2} \mathbb{P} \big( H^u(\tfrac{9}{8} w(h), 3 h) \big) > \frac{\gamma\uc{c:ellipticity}^{38}}{64}.
  \end{equation}
\end{lemma}

Intuitively speaking, the above lemma tells us that under \eqref{assemi1} horizontal RSW property holds.
Later we will treat the complementary case.

\begin{proof} We first define, for $u \in \mathbb{R}^2$, and $0< z < w$, the event
  \begin{equation}
    \label{e:green_one}
    \greenone{w}{z}{h}{u} :=
    \theta_u \circ
    \Big[
      L(w, h) \underset{B(w, h)|_R}\longrightarrow [w - z, w] \times \{h\}
    \Big].
  \end{equation}

  In the above event, no random walk trajectory starting either from the bottom or from the right boundary of $B(w,h)$ reaches the the top boundary of the box at a distance larger than $z$ from the right corner.

  The first step of the proof will be to show that
  \begin{equation}
    \label{e:do_bounce}
    \mathbb{P}
    \bigg(
    \greenone{w(h) + 4}{w(h)/8}{h + 12}{u^*}
    \bigg)
    \geq \uc{c:ellipticity}^{8} \frac{\gamma}{4}.
  \end{equation}
  To see why this holds, we define the non-bouncing version of the event appearing in \eqref{e:green_one}
  \begin{equation}
    \neasttwo{w}{z}{h}{u} :=
    \theta_u \circ
    \Big[
      L(w, h) \underset{B(w, h)}\longrightarrow [w - z, w] \times \{h\}
    \Big]
  \end{equation}
  and observe that it is contained in the event in \eqref{e:green_one}.
  Therefore, in order to prove \eqref{e:do_bounce}, we can assume
  \begin{equation}
    \label{assemi2}
    \mathbb{P}
    \bigg(
      \neasttwo{w(h) + 4}{w(h)/8}{h + 12}{u^*}
    \bigg)
    \leq \frac{\gamma}{2},
  \end{equation}
  because, otherwise the left-hand side in \eqref{e:do_bounce} would be bounded from below by $\gamma /2$ which is greater than the desired bound.
  So now our intermediate objective is to prove \eqref{e:do_bounce} assuming \eqref{assemi2}.
  First define
  \begin{equation}
    \vdotdotdot{w}{z}{h}{u} := \theta_u \circ \bigg[
    \{(w/2, 0)\} \underset{[0, w] \times [0, h]}\longrightarrow
    [0, z] \times \{h\}
    \bigg].
  \end{equation}
  and observe that for every $u \in \mathbb{R}^2$,
  \begin{equation}
    \label{e:figurinhas}
    \vdotdot{w}{z}{h}{u} \setminus \neasttworef{w}{z}{h}{u} \subseteq \vdotdotdot{w}{z}{h}{u}.
  \end{equation}
  Indeed, on the event appearing in the left hand side, any path that entails the occurrence of the vertical crossing cannot touch the right side of the box.

  Using Assumptions~\eqref{assemi1} and~\eqref{assemi2} above and symmetry, \eqref{e:figurinhas} implies that
  \begin{equation}\label{eq:prob_1}
    \mathbb{P}
    \bigg(
    \vdotdotdot{w(h)+4}{w(h)/8}{h+12}{u^*}
    \bigg)
    \geq
    \mathbb{P}
    \bigg(
    \vdotdot{w(h)+4}{w(h)/8}{h+12}{u^*}
    \bigg)
    -
    \mathbb{P}
    \bigg(
    \neasttworef{w(h)+4}{w(h)/8}{h+12}{u^*}
    \bigg)
    \geq
    \frac{\gamma}{2},
  \end{equation}
  and, by symmetry we have
  \begin{equation}\label{eq:prob_12}
    \mathbb{P}
    \bigg(
    \vdotdotdotref{w(h)+4}{w(h)/8}{h+12}{u^*}
    \bigg)
    \geq \frac{\gamma}{2}.
  \end{equation}

 Now define
  \begin{equation}
    \vvdot{w}{z}{h}{u} :=
    \theta_u \circ
    \bigg[
    \{(w/2,0)\} \underset{B(w, h)|_R}\longrightarrow [w - z, w] \times \{h\}
    \bigg],
  \end{equation}
so that we also have
  \begin{equation}
    \label{e:vvdot_lower}
    \mathbb{P}
    \bigg(
    \vvdot{w(h)+4}{w(h)/8}{h+12}{u^*}
    \bigg)
    \geq \frac{\gamma}{2}.
  \end{equation}
  since the event above contains the one appearing in \eqref{eq:prob_12} but has the advantage of being monotone increasing.
  By the FKG inequality~\eqref{e:fkg} together with~\eqref{e:vvdot_lower},~\eqref{e:horizontal_lower}, and uniform ellipticity, we obtain
  \begin{equation}
    \mathbb{P}
    \bigg(
    \greenone{w(h) + 4}{w(h)/8}{h + 12}{u^*}
    \bigg)
    \geq
    \mathbb{P}
    \bigg(
    \vvdot{w(h) + 4}{w(h)/8}{h + 12}{u^*}
    \mcap \;\;
    H^{u^*}(w(h) + 4, h + 12)
    \bigg)
    \overset{\eqref{e:vvdot_lower}}
    \geq \uc{c:ellipticity}^{8} \frac{\gamma}{4}.
  \end{equation}
  which proves~\eqref{e:do_bounce}.

  We consider the intersection of events illustrated in Figure~\ref{fig:crazy_intersection}.
  Fix $v = (x,h+12)$ such that $x+\tfrac{1}{2}w(h)+2$ is an integer in the interval $\big[\tfrac{3}{4}w(h), \frac{7}{8}w(h)+4 \big]$, and define
  \begin{equation}\label{eq:crazy_intersection}
    \greenone{w(h) + 4}{w(h)/8}{h+12}{u} \; \mcap \;\;
    \dot{V}^{u+v}(w(h)+4, h+12) \;\; \cap \;\; H^{u+v}(w(h)+4, h+12),
  \end{equation}
  and notice that, in the event above, the rectangle $u+\big[ 0, \tfrac{5}{4}w(h)+4 \big] \times [0, 2h+24]$ has a horizontal crossing (see Figure~\ref{fig:crazy_intersection}).

\begin{figure}
  \centering
  \begin{tikzpicture}[scale=0.8]
    \draw[thick] (0, 0) rectangle (2.6, 8);

    \draw[color=green!70!black, thick] (0, 0) rectangle (2, 4);
    \draw (0, .3) .. controls (0.4, .5) and (1, 1) .. (2, 1.3);
    \draw (2, 1.3) .. controls (1.4, 1.8) .. (2, 2.6);
    \draw (2, 2.6) .. controls (.8, 3.3) .. (1.85, 4);
    \draw[very thick, color=black] (1.6, 4) -- (2, 4);
    \fill[pattern=north west lines, pattern color=green!70!black] (1.6, 3.9) rectangle (2,4.1);
    \draw[color=green!70!black] (1.6, 3.9) rectangle (2,4.1);
    \draw[shift={(0.6,4)}, color=purple, thick] (0, 0) rectangle (2, 4);
    \draw[shift={(0.6,4)}] (1, 0) .. controls (0.4, .5) and (1, 1) .. (2, 1.3);
    \draw[shift={(0.6,4)}] (2, 1.3) .. controls (1.4, 1.8) .. (2, 2.6);
    \draw[shift={(0.6,4)}] (2, 2.6) .. controls (.8, 3.3) .. (1.5, 4);
    \draw[shift={(0.6,4)}] (0, 1.3) .. controls (0.5,2.7) and (1.7, 2.5) .. (2, 2.6);
    \draw[thick, blue] (0, 0.1) .. controls (2.5, 0.3) and (2.6,0.5) .. (2.3,1) .. controls (2.2,1.2) and (2,1.5) .. (2.3, 2);
    \draw[thick, blue] (2.3, 2) .. controls (2.6,2.6) and (2.6,2.6) .. (2.3, 3.5) .. controls (2.2,4) .. (2.6, 4.5);

    \fill (0.6, 4) circle (.07);
    \fill (0, 0) circle (.07);
    \node at (-0.2, -0.2) {$u$};
    \node[anchor=north] at (0.6, 4) {\small $u+v$};

    \draw[<->] (0, -0.4) -- (2, -0.4);
    \node[anchor=north] at (1, -0.4) {\small $w(h) + 4$};
    \draw[dotted, thick] (2, 0) -- (2, -0.6);

    \draw[<->] (0.6, 8.4) -- (2.6, 8.4);
    \draw[<->] (0, 9.2) -- (0.6, 9.2);
    \node[anchor=south] at (1.6, 8.4) {\small $w(h) + 4$};
    \node[anchor=south] at (0.3, 9.2) {\small $\frac{1}{8} w(h)$};
    \draw[dotted, thick] (0.6, 8) -- (0.6, 9.4);
    \draw[dotted, thick] (2.6, 8) -- (2.6, 8.6);
    \draw[dotted, thick] (0, 8) -- (0, 9.4);

    \draw[<->] (0, -1.2) -- (2.6, -1.2);
    \node[anchor=north] at (1.3, -1.2) {\small $\frac{9}{8}w(h)+4$};
    \draw[dotted, thick] (0, 0) -- (0, -1.4);
    \draw[dotted, thick] (2.6, 0) -- (2.6, -1.4);

    \draw[<->] (-0.4, 0) -- (-0.4, 4);
    \node[anchor=east] at (-0.4, 2) {\small $h+8$};
    \draw[dotted, thick] (0, 0) -- (-0.6, 0);
    \draw[dotted, thick] (0, 4) -- (-0.6, 4);

    \draw[<->] (3, 4) -- (3, 8);
    \node[anchor=west] at (3, 6) {\small $h + 8$};
    \draw[dotted, thick] (2.6, 4) -- (3.2, 4);
    \draw[dotted, thick] (2.6, 8) -- (3.2, 8);

  \end{tikzpicture}
  \caption{The intersection in~\eqref{eq:crazy_intersection}. Notice that the larger rectangle must have a  horizontal (blue) crossing in this intersection.}
  \label{fig:crazy_intersection}
\end{figure}
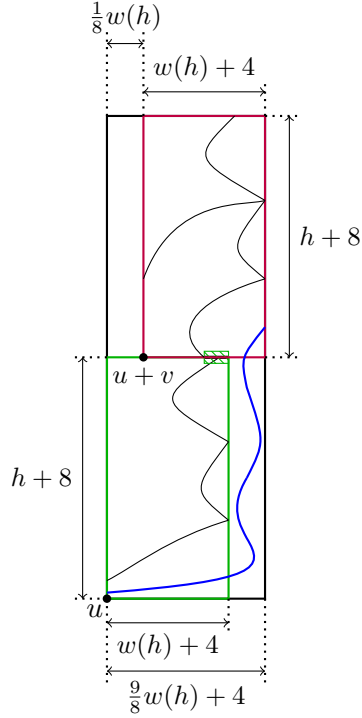

  Furthermore, due to \eqref{e:fkg}, Lemma~\ref{l:lower_v_dot}, uniform ellipticity, and~\eqref{e:do_bounce}, we readily obtain
  \begin{equation}
    \begin{split}
      & \mathbb{P} \big(H^u(\tfrac{5}{4}w(h)+4, 2h+24)\big)\\
      & \geq \mathbb{P} \bigg( \greenone{w(h)+4}{w(h)/8}{h+12}{u} \bigg) \mathbb{P} \big( \dot{V}^{u+v}(w(h)+4, h+12) \big) \mathbb{P} \big(H^{u+v}(w(h)+4, h+12) \big) \\
      & \overset{h \geq \uc{c:large_h}(\uc{c:ellipticity})}{\geq} \frac{\gamma\uc{c:ellipticity}^{8}}{4} \frac{\uc{c:ellipticity}^{17}}{8}\frac{\uc{c:ellipticity}^{9}}{2} = \frac{\gamma\uc{c:ellipticity}^{34}}{64}.
    \end{split}
  \end{equation}

  Finally, we have to relax the condition $\pi_1(u)+w(h)/2\in\mathbb{Z}$. Note that~\eqref{e:sup_inf_H} implies
  \begin{equation}
    \begin{split}
      &\inf_{u \in \mathbb{R}^2} \mathbb{P} \big( H^u(\tfrac{9}{8} w(h), 3h) \big) \geq \sup_{u \in \mathbb{R}^2} \mathbb{P} \big( H^u(\tfrac{9}{8}w(h)+4, 3h-4) \big) \\
      &\geq \sup_{u \in \mathbb{R}^2} \mathbb{P} \big( H^u(\tfrac{5}{4} w(h)+4, 2 h+24) \big) > \frac{\gamma\uc{c:ellipticity}^{34}}{64},
    \end{split}
  \end{equation}
  whenever $h \geq 64 \wedge \uc{c:large_h}(\uc{c:ellipticity})$.
    This concludes the proof.
\end{proof}
The next lemma treats the missing case of the proof of Proposition~\ref{p:h_rsw}.
Note that this is the only part that requires the hypothesis on $w(h)$, which is required for a proper conditional decoupling.

\begin{lemma}
  \label{l:hard_horizontal_rsw}
  Fix $\xi \in (0,1)$ and $\uc{c:alpha}> 1+\xi^{-1}$.
  There exists $\uc{c:h_large_rsw} = \uc{c:h_large_rsw}(\uc{c:ellipticity}, \uc{c:bad_coupling}, \uc{c:alpha}, \xi) > 0$ such that, for all $h \geq \uc{c:h_large_rsw}$ the following holds.
  If
  \begin{equation}
    \label{e:bootstrap_hypothesis}
    w(h) \geq h^{\xi},
  \end{equation}
  and for all $u \in \mathbb{R}^2$ such that $\pi_1(u) + w(h)/2 \in \mathbb{Z}$, we have
  \begin{equation}\label{bigassumption}
    \mathbb{P} \bigg( \vdotdot{w(h)+4}{w(h)/8}{h+12}{u} \bigg) \le \gamma := \frac{\uc{c:ellipticity}^{16}}{16},
  \end{equation}
  then
  \begin{equation}
    \label{e:hard_horizontal_rsw}
   \inf_{u\in\mathbb{R}^2} \mathbb{P} \big( H^u \big( (1+ \tfrac{1}{128})w(h), 3h \big) \big) \geq \frac{\uc{c:ellipticity}^{24}}{256}.
  \end{equation}
\end{lemma}

\begin{proof}
  \nc{c:h_large_rsw}
  We take $\uc{c:h_large_rsw} = \uc{c:h_large_rsw}(\uc{c:ellipticity}, \uc{c:bad_coupling}, \uc{c:alpha}, \xi) \geq \uc{c:large_h}(\uc{c:ellipticity})$ large enough so that $w(h) \geq 128$ and
  \begin{equation}
    \label{e:epsilon_tilde_small}
    \varepsilon \Big( \big(1 + \tfrac{1}{32} \big) w(h), 3h, \tfrac{w(h)}{32} \Big) < \frac{\uc{c:ellipticity}^{24}}{256}, \;\; \text{for all $h \geq \uc{c:h_large_rsw}$},
  \end{equation}
  which is possible because of \eqref{e:bootstrap_hypothesis}, the fact that $\uc{c:alpha}> 1+\xi^{-1}$, and \eqref{e:polynomial}.

  Observing, by \eqref{e:horizontal_lower}, that for all $u\in\mathbb{R}^2$ and all $h > \uc{c:large_h}(\uc{c:ellipticity})$, we have
\begin{equation}
  \label{Hencore}
  \begin{split}
    \mathbb{P} \big( H^u(w(h) - 4, h + 4) \big)&\ge1/2,\\
    \mathbb{P} \big( H^u(w(h) + 4, h + 12) \big)&\ge \uc{c:ellipticity}^8/2,
  \end{split}
\end{equation}
where we used the uniform ellipticity stated in \eqref{e:unif_ellipticity_0} for the second inequality.

Let us recall also that, by Lemma \ref{l:lower_v_dot}, for all $h \geq \uc{c:large_h}(\uc{c:ellipticity})$ and for every $u \in \mathbb{R}^2$ such that $\pi_1(u) + w(h)/2 \in \mathbb{Z}$, we have
\begin{equation}
  \label{vdotencore}
  \mathbb{P} \big( \dot{V}^u(w(h) + 4, h - 4) \big) > 1/8.
\end{equation}
Moreover, uniform ellipticity from~\eqref{e:unif_ellipticity_0} once again allows us to deduce that, for every $h \geq \uc{c:large_h}(\uc{c:ellipticity})$ and $u \in \mathbb{R}^2$ such that $\pi_1(u) + w(h)/2 \in \mathbb{Z}$,
\begin{equation}\label{vdotplus}
    \mathbb{P} \big( \dot{V}^u(w(h) + 4, h + 12) \big) > \uc{c:ellipticity}^{16}/8.
\end{equation}

Observe now that for all $h \geq \uc{c:large_h}(\uc{c:ellipticity})$ and for every $u \in \mathbb{R}^2$ such that $\pi_1(u) + w(h)/2 \in \mathbb{Z}$,
\begin{equation}\label{dansecondclaim}
  \mathbb{P} \bigg( \vdotdotprime{w(h)+4}{w(h)/8}{h+12}{u} \bigg) =
  \mathbb{P} \bigg(\vdot{w(h)+4}{h+12}{u} \;\; \setminus \;\; \vdotdot{w(h)+4}{w(h)/8}{h+12}{u} \bigg) \ge \frac{\uc{c:ellipticity}^{16}}{8}- \gamma \overset{\eqref{bigassumption}}= \frac{\uc{c:ellipticity}^{16}}{16},
\end{equation}
which follows from \eqref{vdotplus} and \eqref{bigassumption}.
Observe also that the event in \eqref{dansecondclaim} above is monotone in the sense of \eqref{eq:increasing}.

Having collected these observations, we are now ready to prove the main statement of the lemma.
Recalling that $w(h) \geq 256$,~\eqref{e:cross_contained} yields
\begin{equation}
\inf_{u\in\mathbb{R}^2} \mathbb{P} \big( H^u \big( (1+ \tfrac{3}{256})w(h), 3h \big) \big) \geq \inf_{\substack{u \in \mathbb{R}^2;\\ \pi_1(u) + w(h)/2 \in \mathbb{Z}}} \mathbb{P} \big( H^u \big( (1+ \tfrac{1}{64})w(h), 3h \big) \big),
\end{equation}
from which it suffices to verify that
\begin{equation}
\inf_{\substack{u \in \mathbb{R}^2;\\ \pi_1(u) + w(h)/2 \in \mathbb{Z}}} \mathbb{P} \big( H^u \big( (1+ \tfrac{1}{64})w(h), 3h \big) \big) \geq \frac{\uc{c:ellipticity}^{24}}{256}.
\end{equation}
Fix
\begin{equation}
  \label{e:eta}
  \eta := \frac{1}{32}
\end{equation}
and define the box
\begin{equation*}
  B = u + [0, \tilde{w}] \times [0, \tilde{h}],
\end{equation*}
where $\tilde{w}=(1+\eta)w(h)-4$ and $\tilde{h}=2h+16+w(h)/32$.
Note that
\begin{itemize}
\item $\tilde{w}$ is roughly $(1 + \eta)$ times larger than $w(h)$,
\item $\tilde{h}$ is smaller or equal to $3 h$.
\end{itemize}
These two observations together guarantee us that crossing $B$ horizontally, compared to the well-balanced box $w(h) \times h$, provides an actual macroscopic improvement, similar to the RSW result in \cite{D-CTT18}.

\medskip

We are now going to consider the decoupling provided by Definition~\ref{def:decoupling}, with $C = B$ and
\begin{equation}
  \label{e:r_spacing}
  r = \frac{w(h)}{32}.
\end{equation}
This decoupling provides us with a field $f^{B, r}$ that is very likely to coincide with our original environment (see \eqref{e:likely_equal}) and at the same time has short range of dependencies inside $B$, see~\eqref{e:finite_range}.

Observe that our hypothesis \eqref{bigassumption} is stated for every $u \in \mathbb{R}^2$ such that $\pi_1(u) + w(h)/2 \in \mathbb{Z}$.
However, the conclusion \eqref{e:hard_horizontal_rsw} of the lemma gives a bound that is uniform over all $u \in \mathbb{R}^2$.
In order to fix this discrepancy, we use \eqref{e:cross_contained} to obtain that
\begin{display}
  \label{e:nest_tilde_box}
  for every $u \in \mathbb{R}$, there is $u' \in \mathbb{R}^2$ satisfying $\pi_1(u') + w(h)/2 \in \mathbb{Z}$\\
  and such that $H^{u'}(\tilde{w}, \tilde{h})$ is contained in $H^u \big( (1 + 1/64)w(h), 3h \big)$,
\end{display}
which is possible because $\tilde{w} \geq (1 + 1/64) w(h) + 10$ and $\tilde{h} \leq 3h - 10$.

We now claim that
\begin{display}
  \label{e:cross_B_enough}
  if $\P_{f^{B, r}} \big( H^u(\tilde{w}, \tilde{h}) \big) \geq \frac{\uc{c:ellipticity}^{24}}{128}$ for all $u \in \mathbb{R}^2$ such that $\pi_1(u) + w(h)/2 \in \mathbb{Z}$, then the desired bound \eqref{e:hard_horizontal_rsw} will hold.
\end{display}
Indeed, if we assume $\P_{f^{B, r}} \big( H^u(\tilde{w}, \tilde{h}) \big) \geq {\uc{c:ellipticity}^{24}}/{128}$, then we can estimate
\begin{equation}
  \label{e:derive_cross_B_enough}
  \begin{split}
    \inf_{\substack{u \in \mathbb{R}^2;\\ \pi_1(u) + w(h)/2 \in \mathbb{Z}}} \mathbb{P} \big( H^u \big( (1 + \tfrac{1}{64}) w(h), 3h \big) \big)
    & \overset{\eqref{e:nest_tilde_box}}\geq \inf_{\substack{u' \in \mathbb{R}^2;\\ \pi_1(u') + w(h)/2 \in \mathbb{Z}}} \mathbb{P} \big( H^{u'} (\tilde{w}, \tilde{h}) \big)\\
    & \overset{\eqref{e:epsilon_tilde_small}}\geq \inf_{\substack{u' \in \mathbb{R}^2;\\ \pi_1(u') + w(h)/2 \in \mathbb{Z}}} \P_{f^{B, r}} \big( H^{u'} (\tilde{w}, \tilde{h}) \big) - \frac{\uc{c:ellipticity}^{24}}{256}\\
    & \geq \frac{\uc{c:ellipticity}^{23}}{256},
  \end{split}
\end{equation}
hence \eqref{e:hard_horizontal_rsw} holds and the claim \eqref{e:cross_B_enough} is proved.  Let us now prove that
\begin{equation}
  \label{e:P_f}
  \P_{f^{B, r}} \big( H^u(\tilde{w}, \tilde{h}) \big) \geq \frac{\uc{c:ellipticity}^{24}}{128}.
\end{equation}

We first consider a sub-box $B_{\text{top}}$ that lies at the top of $B$:
\begin{equation}
  \label{e:B_top}
  B_{\text{top}} = u + (\eta w(h), \tilde{h} - h - 4) \times [w(h) - 4, h + 4].
\end{equation}
see the red box in Figure~\ref{fig:crazy_intersection}.
We consider the event that it is crossed horizontally:
\begin{equation}
  \label{e:H_top}
  H_{\text{top}} = \Big\{B_{\text{top}} \text{ is crossed horizontally} \Big\}
\end{equation}
and observe that for the modified field $f^{B, r}$ we have
\begin{equation}
  \label{e:cross_B_top}
  \P_{f^{B, r}} \big( H_{\text{top}} \big) \geq \frac{1}{2} - \varepsilon(\tilde{w}, \tilde{h}, r),
\end{equation}
by \eqref{e:horizontal_lower} and \eqref{e:likely_equal}.

Note also that $B_{\text{top}}$ is located in the top-right corner of $B$, so that we will use the horizontal crossing defining $H_{\text{top}}$ as the last piece in the horizontal crossing of $B$.

An important part of our argument now is to explore $B_{\text{top}}$ from top to bottom until we are able to witness the horizontal crossing.
More precisely let
\begin{equation}
  \label{e:explore_B_top}
  S = \inf \Big\{ s \in \mathbb{R}; B_{\text{top}} \cap \mathbb{R} \times [s, \infty) \text{ is crossed horizontally} \Big\},
\end{equation}
where we implicitly assume that $\inf \varnothing = - \infty$, noting that $H_{\text{top}} = \left\{S > -\infty\right\}$.

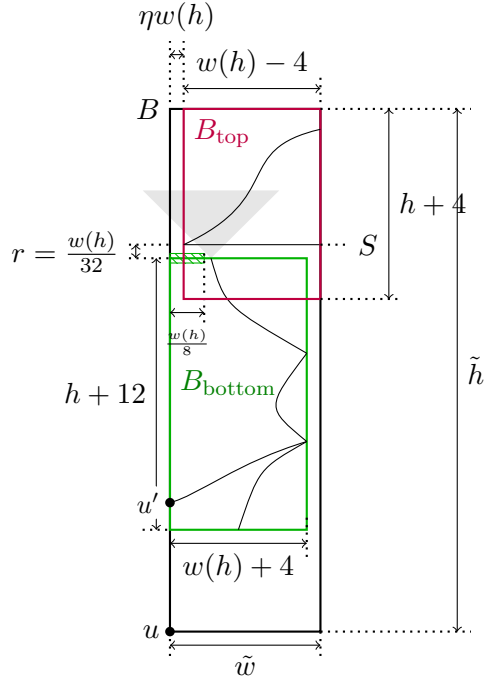
\begin{figure}
  \centering
  \begin{tikzpicture}[scale=.9]
    \fill (0, .5) circle (.07);
    \node[anchor=east] at (0, .5) {$u$};
    % cone
    \fill[opacity=.1] (.6, 6) -- (1.6, 7) -- (-.4, 7) -- (.6, 6);
    % B
    \draw[thick] (0, .5) rectangle (2.2, 8.2);
    \draw (0, 8.2) node[anchor=east] {$B$};
    % B_bottom
    \draw[shift={(0,2)},color=green!70!black,thick] (0, 0) rectangle (2, 4);
    \draw[shift={(0,2)}] (0, 2.5) node[anchor=north west,color=green!50!black] {$B_{\text{bottom}}$};

    \draw[dotted, thick] (-.4, 2) -- (0, 2);
    \draw[<->] (-.2, 2) -- (-.2, 6);
    \node[anchor=east] at (-.2, 4) {$h + 12$};
    % crossings in bottom
    \fill[color=white] (-.3, 2.4) circle (.26);
    \fill (0, 2.4) circle (0.07);
    \node[anchor=east] at (0, 2.4) {\small $u'$};
    \draw[shift={(0,2)}] (0, .4) .. controls (0.4, .5) and (1, 1) .. (2, 1.3);
    \draw[shift={(0,2)}] (1, 0) .. controls (1.3, .9) .. (2, 1.3);
    \draw[shift={(0,2)}] (2, 1.3) .. controls (1.4, 1.8) .. (2, 2.6);
    \draw[shift={(0,2)}] (2, 2.6) .. controls (.8, 3.3) .. (.6, 4);
    \fill[pattern=north west lines, pattern color=green!70!black] (0,5.93) rectangle (.5,6.07);
    \draw[color=green!70!black] (0,5.93) rectangle (.5,6.07);
    \draw[<->] (0, 5.2) -- (.5, 5.2);
    \node[anchor=north] at (.25, 5.1) {\tiny $\tfrac{w(h)}{8}$};
    \draw[dotted, thick] (.5, 6.07) -- (.5, 5);
    % B_top
    \draw[shift={(0.2,6.2)}] (0, 0) rectangle (2, 2);
    \draw[shift={(0.2,6.2)},color=purple,thick] (0, -.8) rectangle (2, 2);
    \draw[shift={(0.2,6.2)}] (0, 2) node[anchor=north west,color=purple] {$B_{\text{top}}$};
    % crossing in B_top
    \draw[shift={(0.2,6.2)}] (0, 0) .. controls (1.5,.7) and (.7, 1.4) .. (2, 1.7);

    \draw[dotted, thick] (2.2, 6.2) -- (2.6, 6.2);
    \node[anchor=west] at (2.6, 6.2) {$S$};

    \draw[dotted, thick] (-.6, 6) -- (0, 6);
    \draw[dotted, thick] (-.6, 6.2) -- (0, 6.2);
    \draw[<->] (-.5, 6) -- (-.5, 6.2);
    \draw (-.6, 6.1) node[anchor=east] {$r = \frac{w(h)}{32}$};

    \draw[<->] (0, 1.8) -- (2, 1.8);
    \node[anchor=north] at (1, 1.8) {$w(h) + 4$};
    \draw[dotted, thick] (2, 1.6) -- (2, 2.2);

    \draw[<->] (.2, 8.5) -- (2.2, 8.5);
    \node[anchor=south] at (1.2, 8.5) {$w(h) - 4$};
    \node[anchor=south] at (.1, 9.2) {$\eta w(h)$};
    \draw[dotted, thick] (.2, 8.3) -- (.2, 9.2);
    \draw[dotted, thick] (0, 8.3) -- (0, 9.2);
    \draw[<->] (0, 9) -- (.2, 9);
    \draw[dotted, thick] (2.2, 8.3) -- (2.2, 8.7);

    \draw[<->] (0, .3) -- (2.2, .3);
    \node[anchor=north] at (1.1, .3) {$\tilde{w}$};
    \draw[dotted, thick] (0, .5) -- (0, .1);
    \draw[dotted, thick] (2.2, .5) -- (2.2, .1);

    \draw[<->] (4.2, .5) -- (4.2, 8.2);
    \node[anchor=west] at (4.2, 4.35) {$\tilde{h}$};
    \draw[dotted, thick] (2.2, .5) -- (4.4, .5);
    \draw[dotted, thick] (2.2, 8.2) -- (4.4, 8.2);

    \draw[<->] (3.2, 5.4) -- (3.2, 8.2);
    \node[anchor=west] at (3.2, 6.8) {$h + 4$};
    \draw[dotted, thick] (2.2, 5.4) -- (3.4, 5.4);
  \end{tikzpicture}
  \caption{The intersection of events in~\eqref{e:crazy_intersection_2}.}
  \label{fig:crazy_intersection_2}
\end{figure}

Consider now a configuration for which $\left\{S > -\infty\right\}$.
Informally speaking, we are going to position a box $B^S_{\text{bottom}}$ underneath the crossing of $B_{\text{top}}$.
This box will have same dimensions as the box appearing in \eqref{bigassumption} and it will be put within distance $r$ from the explored region.
More precisely, let
\begin{equation}
  \label{e:B_bottom}
  B^S_{\text{bottom}} = u + [0, w(h) + 4] \times [S - r - (h + 12), S - r],
\end{equation}
and refer again to Figure~\ref{fig:crazy_intersection} for an illustration of $B^S_{\text{bottom}}$ represented there in green.
Analogously we consider
\begin{equation}
  \label{e:H_bottom}
  H^S_{\text{bottom}} = \Big\{ B^S_{\text{bottom}} \text{ is crossed horizontally} \Big\}.
\end{equation}

We are now in position to describe the intersection of events that guarantee the horizontal crossing of $B$.
First let us give an intuitive description of them:
\begin{enumerate}[\quad a)]
\item We assume that $H_{\text{top}}$ occurred, or in other words, $S > -\infty$;
\item Having positioned $B^S_{\text{bottom}}$, we will find a point $u'$ on its left face, so that when the random walk starting from $u'$ reaches height $S - r$, it will have moved at least $w(h)/8$ to the right, see Figure~\ref{fig:crazy_intersection}.
\item Between heights $S - r$ and $S$, the walk will not have enough time (by its Lipschitz character) to move to the left of the horizontal crossing of $B_{\text{top}}$,
\end{enumerate}
therefore the walk starting at $u'$ will be forced by monotonicity to cross $B$ horizontally as claimed in \eqref{e:P_f}.

Having given the informal description of our construction, let us now define precisely the intersection of events we consider:
\begin{equation}
  \label{e:crazy_intersection_2}
  \mathcal{A} \;\; := \;\; H_{\text{top}} \; \cap \; H^S_{\text{bottom}} \; \cap \; \vdotdotprime{w(h)+4}{w(h)/8}{h+12}{B^S_{\text{bottom}}}.
\end{equation}
What is left for us to do is to prove that
\begin{align}
  \label{e:bound_crazy_intersection}
  & \P_{f^{B, r}} \big( \mathcal{A} \big) \geq \frac{\uc{c:ellipticity}^{24}}{128} \text{ and}\\
  \label{e:crazy_then_crosses}
  & \mathcal{A} \subseteq H^u(\tilde{w}, \tilde{h}),
\end{align}
which clearly imply \eqref{e:P_f}.

We start by proving \eqref{e:bound_crazy_intersection}.
Using the fact that the model is invariant under vertical translations, that $f^{B, r}$ is $r$-dependent in $B$ (see \eqref{e:finite_range}) and $\d(B_{\text{top}}, B^S_{\text{bottom}}) \geq r$, we obtain
\begin{equation*}
  \begin{split}
    \P_{f^{B, r}} (\mathcal{A})
    & = \P_{f^{B, r}} (H_{\text{top}}) \;\;
      \P_{f^{B, r}} \bigg( H^u \big( w(h) + 4, h + 12 \big) \; \mcap \; \vdotdotprime{w(h)+4}{w(h)/8}{h+12}{u} \bigg)\\
    & \geq \P_{f^{B, r}} (H_{\text{top}}) \;\;
     \bigg(  \P \bigg( H^u \big( w(h) + 4, h + 12 \big) \; \mcap \; \vdotdotprime{w(h)+4}{w(h)/8}{h+12}{u} \bigg) - \varepsilon(\tilde{w}, \tilde{h}, r) \bigg)\\
    & \overset{\eqref{e:fkg}}\geq \P_{f^{B, r}} (H_{\text{top}}) \;\;
      \bigg( \P \Big( H^u \big( w(h) + 4, h + 12 \big) \Big) \; \;
      \P \bigg( \vdotdotprime{w(h)+4}{w(h)/8}{h+12}{u} \bigg) - \varepsilon(\tilde{w}, \tilde{h}, r) \bigg)\\
    & \geq \Big( \frac{1}{2} - \varepsilon(\tilde{w}, \tilde{h}, r) \Big)
      \frac{\uc{c:ellipticity}^8}{2} \Big( \Big( \frac{\uc{c:ellipticity}^{16}}{8} - \gamma \Big) - \varepsilon(\tilde{w}, \tilde{h}, r) \Big)
      \geq \frac{\uc{c:ellipticity}^{24}}{128},
  \end{split}
\end{equation*}
by \eqref{e:cross_B_top}, \eqref{Hencore} and  \eqref{dansecondclaim}, proving \eqref{e:bound_crazy_intersection}.

All we are left to prove now is \eqref{e:crazy_then_crosses}.
For this, suppose that we are on the event $\mathcal{A}$ and let $u'$ be the starting point of the crossing in $H^S_{\text{bottom}}$.
It is clear that the horizontal crossing of $B^S_{\text{bottom}}$ from $x$ has to cross the bouncing path in the last event defining $\mathcal{A}$.
Therefore, when $X^{u'}$ reaches the height $S - r$, it will either have crossed $B$ horizontally (in which case we are done), or it will be to the right of $\pi_1(u) + w(h)/8$.
In this case, by the Lipschitz property of $X^{u'}$, when this wall reaches height $S$, it will be to the right of $\pi_1(u) + w(h)/8 - r$, which in turn is inside the box $B_{\text{top}} \cap \mathbb{R} \times [S, \infty)$.
This means that $X^{u'}$ will stay to the right of the horizontal crossing of $B_{\text{top}}$ and therefore it will also cross $B$, proving \eqref{e:crazy_then_crosses}, see Figure~\ref{fig:crazy_intersection}.

This concludes the proof of \eqref{e:P_f} and finishes the proof of the lemma.
\end{proof}

This concludes the proof of Proposition~\ref{p:h_rsw}.

\section{Bootstrapping fluctuations and proof of Theorem~\ref{t:main}}
\label{s:bootstrap}

The Russo-Seymour-Welsh estimates proved in the previous section allow us to obtain the following bootstrapping bound on the fluctuations.

%%%%%
\nc{c:bootstrap}
\nc{c:bootstrap0}
\nc{c:fluctuation}
%%%%%
\begin{lemma}
  \label{l:bootstrap}
  Assume that $\uc{c:ellipticity} \geq \frac{1}{4}$. For any fixed $\xi_{0} \in (0,1)$ there exist a decay rate $\uc{c:alpha}$, $\xi \in (0,\xi_0]$ and $\uc{c:bootstrap0} = \uc{c:bootstrap0}(\uc{c:bad_coupling}, \xi_{0})$ such that, if
  \begin{equation}
    \label{e:lower_bound_w_h0}
    w(h_{0}) \geq h_{0}^{\xi},
  \end{equation}
  for some $h_0 \ge \uc{c:bootstrap0}$, then
  \begin{equation}
    w(h) \geq \uc{c:fluctuation}h^{\xi},
  \end{equation}
  for all $h \geq h_{0}$.
\end{lemma}

\begin{proof}
  To emphasize that there is no cyclic choice of constants inside the proof, we present their choice beforehand.
  Although it should be noted that the exact values of these constants is not important and other values would make the proof go through as well. Start by fixing $\uc{c:ellipticity} = \tfrac{1}{4}$ and set
  \begin{equation}\label{defn_xi}
    n := \left\lceil \frac{\ln(8)}{-\ln(1-2^{-118})}\right\rceil \quad \text{and} \quad \xi := \frac{\ln(129/128)}{\ln(4n)} \wedge \xi_{0}.
  \end{equation}
  Choose now $\alpha > 1+\xi^{-1} >2$. Finally, set
  \begin{equation}
    \label{defh0}
    \uc{c:bootstrap0} = \uc{c:bootstrap0} (\uc{c:bad_coupling}, \xi_{0}) = \left\lceil \left( 80 \uc{c:bad_coupling} n \right)^{1/(\alpha - 2)} \right\rceil \vee \big( \uc{c:h_large_rsw}(\tfrac{1}{4}, \uc{c:bad_coupling}, \alpha, \xi) + 64 \wedge \uc{c:large_h}(\tfrac{1}{4}) \big),
  \end{equation}
  where $\uc{c:large_h}$ is such that~\eqref{e:horizontal_lower} and~\eqref{e:vertical_lower} hold and $\uc{c:h_large_rsw}$ is given by Lemma~\ref{l:hard_horizontal_rsw}.
  By Proposition~\ref{p:h_rsw}, using that $\uc{c:alpha} > \xi^{-1}+1$ and $h_0 \ge \uc{c:bootstrap0}$, if $w(h_0) \ge h_0^\xi$ then
  \begin{equation}
    \inf_{u\in\mathbb{R}^2}\mathbb{P} \big( H^u \big(\tfrac{129}{128} w(h_0), 3h_0 \big)\big) \ge \frac{\uc{c:ellipticity}^{54}}{2^{10}} \geq 2^{-118}.
  \end{equation}
  Now we use Lemma~\ref{l:stretch} with $h=3h_{0}$ to obtain
  \begin{equation}
    \mathbb{P} \Big( H^u \big( \tfrac{129}{128} w(h_{0}), 4n h_{0})^{c} \Big) \leq \mathbb{P} \Big( H^u \big( \tfrac{129}{128} w(h_{0}), 3 h_{0})^{c} \Big)^{n} + n\varepsilon\big( \tfrac{129}{128}w(h_{0}), 3h_{0}, h_{0} \big).
  \end{equation}
  By Lemma \ref{l:error_decoupling}, we obtain, since $w(h_{0}) \leq h_{0}$,
  \begin{align}
    \label{jesaispas}
    \begin{split}
      \mathbb{P} \Big( H^u \big( \tfrac{129}{128} w(h_{0}), 4n h_{0} \big)^{c} \Big) & \leq
      \mathbb{P} \Big( H^u \big( \tfrac{129}{128} w(h_{0}), 3 h_{0})^{c} \Big)^{n}\\
      & \quad + \uc{c:bad_coupling} n \Big( \tfrac{129}{128} w(h_0) 3 h_0 + \big( \tfrac{129}{128} w(h_0) + 3 h_0 \big) h_0 + h_0^2 \Big) h_0^{-\uc{c:alpha}	}\\
      & \leq \mathbb{P} \Big( H^u \big( \tfrac{129}{128} w(h_{0}), 3 h_{0})^{c} \Big)^{n} + 10 \uc{c:bad_coupling} n h_0^{-\uc{c:alpha}+2}\\
      & \le \big(1-2^{-118} \big)^{n} + 10 \uc{c:bad_coupling} n h_0^{-\uc{c:alpha}+2}\\
      & \overset{\text{(see below)}}\le \frac18+\frac18<\frac12,
    \end{split}
  \end{align}
  where we used for the last line that $\alpha>2$ and that, by \eqref{defn_xi} and \eqref{defh0}, we have
  \begin{align}
    n \ge \frac{\ln(8)}{-\ln(1-2^{-118})}\text{ and } h_0 \ge \big ( 80 \uc{c:bad_coupling} n \big)^{1/(\alpha - 2)}.
  \end{align}
  This implies that
  \begin{align}
    \label{jesaispas2}
      w(4nh_0)
      \ge \frac{129}{128} w(h_0)\ge \frac{129}{128}h_0^\xi
      \ge \frac{129}{128(4n)^\xi} \left(4nh_0\right)^\xi \;\; \overset{\eqref{defn_xi}}{\geq} \;\; \left(4nh_0\right)^\xi.
  \end{align}
  Repeating inductively the steps from \eqref{jesaispas} to \eqref{jesaispas2}, we obtain
  \begin{equation}\label{eq:interpolation_0}
    w \big( (4n)^{k} h_{0} \big) > \big( (4n)^k h_{0} \big)^{\xi}, \quad \text{for all } k \in \N.
  \end{equation}
  We now claim that
  \begin{equation}
    \label{eq:interpolation}
    w(h) \geq \left(\frac{h}{4n}\right)^{\xi}, \quad \text{for all } h \geq h_{0},
  \end{equation}
  which concludes the proof with the choice $\uc{c:fluctuation} = (4n)^{-\xi}$.

  In order to verify~\eqref{eq:interpolation}, note that for all $h\ge h_0$, there exists a unique $k\in \mathbb{N}$ such that $h \in [(4n)^{k}h_{0}, (4n)^{k+1}h_{0})$ and notice that
  \begin{equation}
    \begin{split}
      \mathbb{P} \Big( H^u \big( \big( \tfrac{h}{4n} \big)^{\xi}, h \big)^{c} \Big)
      & \leq \mathbb{P} \Big( H^u \big( \big( \tfrac{h}{4n} \big)^{\xi}, (4n)^{k} h_{0} \big)^{c} \Big) \\
      & \leq \mathbb{P} \Big( H^u \big( \big( \tfrac{(4n)^{k+1}h_0}{4n} \big)^{\xi}, (4n)^{k} h_{0} \big)^{c} \Big) \\
      & \leq   \mathbb{P} \Big( H^u \big( (4n)^{k\xi}h_0^{\xi}, (4n)^{k} h_{0} \big)^{c} \Big).
    \end{split}
  \end{equation}
  This in particular implies
  \begin{equation*}
    \frac{1}{\Vol(\mathcal{L}_1)} \int_{u \in \mathcal{L}_1} \mathbb{P} \Big( H^u \big( \big( \tfrac{h}{4n} \big)^{\xi}, h\big)^{c} \Big) \d u \leq \frac{1}{\Vol(\mathcal{L}_1)} \int_{u \in \mathcal{L}_1} \mathbb{P} \Big( H^u \big( \big( (4n)^{k}h_0 \big)^{\xi}, (4n)^{k} h_{0} \big)^{c} \Big) \d u < \frac{1}{2},
  \end{equation*}
  by~\eqref{eq:interpolation_0}, which concludes the proof.
\end{proof}

We are finally in position to establish our main result.

\begin{proof}[Proof of Theorem~\ref{t:main}]
  Fix $\varepsilon \in \big(0, \tfrac{1}{2} \big)$ and choose $\xi_{0} = \tfrac{1}{2}-\varepsilon$. Let $\uc{c:alpha}>0$, $\xi \in (0, \xi_{0})$, and $\uc{c:bootstrap0}(\uc{c:bad_coupling}, \xi_0)$ be as in Lemma \ref{l:bootstrap}.
  Recall the transition probabilities \eqref{e:X_n}.
  To generate the steps of $(X_n)_{n \ge 0}$, we will consider a sequence $(U_n)_{n \ge 0}$ of i.i.d.~uniform random variables on $[0, 1]$ and declare that $X_{n + 1} - X_n = -1$ if and only if $U_n \le \frac{1}{2} - \delta \big( \textbf{1} \{ f{(X_n, n)}  \geq 0 \} - \textbf{1} \{ f{(X_n, n)} < 0 \} \big)$.
  In particular, note that on the event
  \begin{equation}
    E_{h,\delta}=\bigcap_{n=0}^{h-1}
    \big\{ U_n\in[0,\tfrac12-\delta]\cup(\tfrac12+\delta,1] \big\},
  \end{equation}
  we can couple $(X_n; 0\le n \le h)$ with a simple random walk $(Y_n; 0\le n \le h)$. Hence, for $h \ge 1$, and any $u \in \mathbb{R}^2$ we have
  \begin{equation}
    \begin{split}
      \mathbb{P} \big( V^u(h^{\xi_0},h) \big)
      & \overset{\text{(see below)}}{\leq} \mathbb{P}(E_{h, \delta}^c)
        + \mathbb{P} \Big( |Y_h| \le 2h^{\xi_0} \Big)\\
      & \leq \mathbb{P}(E_{h, \delta}^c)
        + \mathbb{P} \left( \left| \frac{Y_h}{h^{\frac12}} \right| \le 2h^{-\varepsilon} \right)
        \leq 2 h \delta + \mathbb{P} \left( \left| \frac{Y_h}{h^{\frac{1}{2}}} \right| \le 2h^{-\varepsilon}\right),
    \end{split}
  \end{equation}
  where in the first line above, we used the fact that, in order for the random walk to cross the box when reflected on its right face, it needs to be entirely contained in the box $u+[-h^{\xi_0}, 2h^{\xi_0}] \times [0,h]$. This follows from the reflection principle, using the random variables $1-U_{n}$ to determine the random walk $Y$ when it is to the right of the face $u+\{h^{\xi_0}\} \times [0,h]$.

  Using the Central Limit Theorem, we can choose $h_0\ge \uc{c:bootstrap0}$ large enough (depending on $\varepsilon$), then choose  $\delta > 0$ small enough (depending on $h_0$) such that the probability above is smaller than $1/2$ and thus
  \begin{equation}
    w(h_0)\ge h_0^{\xi_0} \geq h_{0}^{\xi}.
  \end{equation}
By Lemma~\ref{l:bootstrap},
  \begin{equation}
    \label{e:w_lower}
    w(h) \geq \uc{c:fluctuation}h^{\xi}, \quad \text{for all } h\geq h_0,
  \end{equation}
  This lower bound on $w(h)$ can be understood as a lower bound on the fluctuations of $X_n$.

  In order to get the exact statement \eqref{e:main} in Theorem~\ref{t:main}, we proceed with a geometric construction.
  By possibly increasing the value of $h_{0}$, we can assume that $w(h) \geq 512$, for all $h \geq h_{0}$. Fix $j = 3 \cdot 128$ and for $n \geq 3j h_{0}$ write $h = \frac{n}{3j}$ and let
  \begin{gather*}
     B_1 = [-w(h)-4, 0] \times [0, n], \quad B_2 = [-w(h)-4, 2w(h)+4] \times [0, n],  \\
     \text{and} \quad B_3 = [w(h), 2 w(h)+4] \times [0, n].
  \end{gather*}
  We can now use crossing events to estimate
  \begin{equation}
    \begin{split}
      \mathbb{P} \Big( X_n \geq \frac{\uc{c:fluctuation}}{(3j)^{\xi}}n^{\xi} \Big)
      & \overset{\eqref{e:w_lower}}\geq \mathbb{P} \big( X_n \geq w(h) \big)
      \geq \mathbb{P} \big( V(B_1) \cap H(B_2) \cap V(B_3) \big)\\
      & \overset{\eqref{e:fkg}}\geq \mathbb{P} \big( V(B_1) \big) \mathbb{P} \big( H(B_2) \big) \mathbb{P} \big( V(B_3) \big)\\
      & \overset{\text{Corol.~\ref{c:stretch_vertically}}}\geq \Big( \frac{\uc{c:vertical_rsw}\uc{c:ellipticity}^{8}}{2}\Big)^{2 \cdot 3j} \mathbb{P} \big( H(B_2) \big)
      \overset{\text{Corol.~\ref{c:stretch_horizontally}}}\geq \Big( \frac{\uc{c:vertical_rsw}\uc{c:ellipticity}^{8}}{2}\Big)^{2 \cdot 3j} \uc{c:c_rsw}^j,
    \end{split}
  \end{equation}
  finishing the proof of the theorem by possibly taking a smaller $\xi$.
\end{proof}

\begin{remark}
  \label{r:h_rsw_hyp}
  Let us now address the perturbative assumption $\delta < \uc{c:perturbative}$ in our main result.
  By looking at Lemma~\ref{l:bootstrap} it becomes clear that the reason why we need a small interaction between the random walk and the random environment is to obtain a lower bound on fluctuations, or more precisely, a lower bound on $w(h_0)$.

  This triggering assumption for the induction was necessary because in Lemma~\ref{l:hard_horizontal_rsw} we can only obtain a horizontal RSW if we have an a priori lower bound on $w(h)$ at the previous scale, see \eqref{e:bootstrap_hypothesis}.

  Digging a bit deeper into Lemma~\ref{l:hard_horizontal_rsw}, we see that \eqref{e:bootstrap_hypothesis} was used to obtain a decoupling in \eqref{e:epsilon_tilde_small}.
  Therefore, at the end of the day, the reason behind our perturbative assumption is to avoid having a very thin and tall box $B_{\text{top}}$.
  This would mean that the search for a horizontal crossing of $B_{\text{top}}$ could potentially generate a negative information in a large region, forcing us to position $B_{\text{bottom}}$ too far down, which would break our argument.

  Note that if a random environment satisfies Definition~\ref{def:decoupling} with $\varepsilon(w, h, r)$ that decays very fast with $h$ (say exponentially), then we do not expect to need a perturbative assumption on $\delta$.
  But it is an open question whether Proposition~\ref{p:h_rsw} holds without \eqref{e:bootstrap_hypothesis} in general.
\end{remark}

\section{Examples}
\label{s:examples}

\subsection{Gaussian fields}
\label{s:gauss_env}
~
\par Fix a function $q:\Z^{2} \to \R_{+}$ not identically null, and that is invariant under vertical reflexions of the plane.
Assume further that there exists $\beta > 2$ such that
\begin{equation}\label{eq:correlation_q}
q(x) \leq \uc{c:correlation_decay}|x|^{-\beta}, \text{ for all } x \in \Z^{2}\setminus{\{o\}},
\end{equation}
and that $q(o) \leq \uc{c:correlation_decay}$.
Consider now a family $\big(W_{x}\big)_{x \in \Z^{2}}$ of i.i.d.\ standard $\normal(0,1)$ random variables and define the random Gaussian field $\big( g_{x} \big)_{x \in \Z^{2}}$ via
\begin{equation}\label{eq:construction_gf}
g_{x} = \sum_{y}q(x-y)W_{y}.
\end{equation}

Given the Gaussian field $g$, consider now the environment given by
\begin{equation}
\label{eq:field}
f(x) = \sign(g_{x}) = \textbf{1}_{\{g_{x} > 0\}} - \textbf{1}_{\{g_{x}<0\}}.
\end{equation}
Observe that $f$ satisfies~\eqref{e:translation} and~\eqref{e:symmetry} by construction. It remains then to verify that it also satisfies~\eqref{e:fkg} and Definition~\ref{def:decoupling} for $\varepsilon$ as in~\eqref{e:polynomial}.

The~\eqref{e:fkg} inequality follows from the following lemma.
\begin{proposition}[FKG inequality for Gaussian fields~\cite{pitt1982}]
  \label{prop:FKG}
  Let $A$ and $B$ be two increasing events depending on finitely many coordinates of $g$. Then
  \begin{equation}\label{e:fkg_gf}
    \tag{FKG}
    \P \big( g \in A \cap B\big) \geq \P \big( g \in A \big) \P \big( g \in B \big).
  \end{equation}
\end{proposition}

%%%%%
\nc{c:bad_coupling}
%%%%%

The verification of Definition~\ref{def:decoupling} with $\varepsilon$ as in~\eqref{e:polynomial} requires a bit more of work and we state it as a lemma.
\begin{lemma}
  \label{l:error_decoupling}
  The field $f$ satisfies the decoupling condition in Definition~\ref{def:decoupling}, with decay rate given by
  \begin{equation}
    \varepsilon(w, h, r) =  \uc{c:bad_coupling}\big(wh+ (w+h)r+r^{2} \big)r^{-\beta+\frac{3}{2}},
  \end{equation}
  where $\uc{c:bad_coupling}$ is a positive constant.
\end{lemma}

In order to prove the lemma, we start by bounding the covariance of the field $g$.

Observe that, for every $x \in \Z^{2}$, $g_{x}$ is a centered Gaussian random variable with variance
\begin{equation}
\label{eq:variance}
\sigma_{q}^2=\sum_{y \in \Z^{2}} q(y)^{2} < \infty,
\end{equation}
since we are assuming that $\beta >2$.
The covariance structure in this field is given by
\begin{equation}
\label{e:covariance_g}
\cov\big(g_{x},g_{y}\big) = \sum_{z \in \Z^{2}}q(x-z)q(y-z) = \sum_{z \in \Z^{2}}q(x-y-z)q(z) = q*q(x-y).
\end{equation}
Due to~\eqref{eq:correlation_q}, we have
\begin{equation}
\begin{split}
  q*q(x) & = \sum_{z: \, |z-x| \geq \frac{1}{2}|x|}q(x-z)q(z) + \sum_{z: \, |z-x| < \frac{1}{2}|x|}q(x-z)q(z) \\
  & \leq 2\sum_{y: \, |y| \geq \frac{1}{2}|x|}\uc{c:correlation_decay}^{2}|y|^{-\beta} \leq \uc{c:correlation_convolution}|x|^{-\beta+2},
\end{split}
\end{equation}
for some constant $\uc{c:correlation_convolution} > 0$.

We now construct finite-range approximations for $g$, by defining, for each $r>0$, the truncated filed $\big(g_{x}^{r}\big)_{x \in \Z^{2}}$ via
\begin{equation}\label{eq:finite_range_approximation}
g^{r}_{x} = \sum_{y: \, |x-y| \leq \frac{r}{2}}q(x-y)W_{y}.
\end{equation}
Notice that, if $|x-y|>r$, then $g_{x}^{r}$ and $g_{y}^{r}$ are independent.

%%%%%
\nc{c:approximation}
%%%%%
\begin{proposition}\label{prop:finite_range_approximation}
There exists a positive constant $\uc{c:approximation}>0$ such that, for every $r \geq 2$, $x \in \Z^{2}$, and $t \geq 0$,
\begin{equation}
\P \big( |g_{x}-g_{x}^{r}| \geq t \big) \leq 2e^{-\uc{c:approximation}t^{2}r^{2\beta-2}}.
\end{equation}
\end{proposition}

\begin{proof}
By translation invariance, it suffices to consider the case where $x=0$.
We now observe that the random variable
\begin{equation}
g_{0}-g_{0}^{r} = \sum_{y: \, |y| > \frac{r}{2}}q(-y)W_{y}
\end{equation}
has the normal distribution with mean $0$ and variance $\sum_{y: \, |y| > \frac{r}{2}}q(y)^{2}$.
Standard tail estimates for Gaussian variables now yield
\begin{equation}
\P \big( |g_{0}-g_{0}^{r}| \geq t \big) \leq 2e^{-\frac{t^{2}}{2 \sum_{y: \, |y| > \frac{r}{2}}q(y)^{2}}}.
\end{equation}
The proof is now complete by combining the estimate above with
\begin{equation}
\sum_{y: \, |y| > \frac{r}{2}}q(y)^{2} \leq \uc{c:correlation_decay}^{2}\sum_{y: \, |y| > \frac{r}{2}}|y|^{-2\beta} \leq 4\uc{c:correlation_decay}^{2}\sum_{n>\frac{r}{2}}n^{-2\beta+1} \leq \frac{2\uc{c:correlation_decay}^{2}}{\beta-1} \Big( \frac{r}{2}-1 \Big)^{-2\beta+2},
\end{equation}
finishing the proof of the proposition.
\end{proof}

Analogously to~\eqref{eq:finite_range_approximation}, we introduce the finite-range approximations of the environment by defining, for every $r \geq 1$,
\begin{equation}
\label{eq:field_approximation}
f^{r}(x) = \textbf{1}_{\{g_{x}^{r} \geq 0\}}-\textbf{1}_{\{g_{x}^{r}<0\}}.
\end{equation}

%%%%%
\nc{c:approximation_environment}
%%%%%
Finally, we provide a result for the environment $f$ which is analogous to Proposition~\ref{prop:finite_range_approximation}.
\begin{proposition}
\label{prop:finite_range_approximations_2}
There exists $\uc{c:approximation_environment}>0$ such that, for every $r \geq 2$ and $x \in \Z^{2}$,
\begin{equation}
\P \big( f(x) \neq f^{r}(x) \big) \leq \uc{c:approximation_environment}r^{-\beta+\frac{3}{2}}.
\end{equation}
\end{proposition}

\begin{proof}
Once again, it suffices to consider the case $x=0$ by translation invariance. If $f(0) \neq f^{r}(0)$, then for every fixed $t \geq 0$, either $|g_{0}-g_{0}^{r}| \geq t$ or $|g_{0}| \leq t$
This yields
\begin{equation}
\P \big( f(0) \neq f^{r}(0) \big) \leq \P \big( |g_{0}-g_{0}^{r}| \geq t \big) + \P\big( |g_{0}| \leq t \big).
\end{equation}
The first term in the right-hand side of the equation above is bounded in Proposition~\ref{prop:finite_range_approximation}. As for the second term, notice that $g_{0} \sim \normal(0, \sigma_{q})$, with $\sigma_{q}$ given by~\eqref{eq:variance}, and then bound the density of $g_{0}$ by $(2\pi\sigma_{q})^{-\frac{1}{2}}$ to obtain
\begin{equation}
\P\big( |g_{0}| \leq t \big) \leq t\sigma_{q}^{-\frac{1}{2}}.
\end{equation}
We now choose $t=r^{-\beta+\frac{3}{2}}$ to conclude the proof.
\end{proof}

\begin{proof}
The statement follows immediately by combining Proposition~\ref{prop:finite_range_approximations_2} and union bounds. Since $f^{r}$ is constructed with $g^{r}$ as in~\eqref{eq:field_approximation}, it is an $r$-dependent field.
\end{proof}

We are can now prove that the field $f$ satisfies the decoupling condition.

\begin{proof}[Proof of Lemma~\ref{l:error_decoupling}]
Recall the definition of $f^{r}$ in~\eqref{eq:field_approximation}. Fix a rectangle $C = [a, a+w] \times [b, b+h] \subset \Z^{2}$ and define the field
\begin{equation}
{f}^{C, r}(x) =
  \begin{cases}
    f^{r}(x), & \text{ if } \d(x,C) \leq r; \\
    f^{\d(x,C)}(x), & \text{ if } \d(x,C) > r.
  \end{cases}
\end{equation}

We first verify that the field defined above indeed satisfies the first statement of the lemma. Let $A \subset C$ and $B \subset \Z^{2}$ such that $\d (A,B) \geq r$. Notice that ${f}^{C,r}\big|_{A} = f^{r} \big|_{A}$ and thus this restriction is determined by the random variables $\big(W_{y}\big)_{\d(y,A) \leq \frac{r}{2}}$. Let us now study ${f}^{C,r} \big|_{B}$: if $x \in B$ is such that $\d(x,C) \leq r$, then ${f}^{C,r}(x)$ is determined by $\big(W_{y}\big)_{\d(y,x) \leq \frac{r}{2}}$. On the other hand, if $\d(x,C) >r$, then ${f}^{C,r}(x)$ is determined by $\big(W_{y}\big)_{\d(y,x) \leq \frac{\d(x,C)}{2}}$. Therefore, regardless of the choice of $B$, ${f}^{C,r} \big|_{A}$ and ${f}^{C,r} \big|_{B}$ are determined by disjoint collections of the family of Gaussian variables $\big( W_{x} \big)_{ x \in \Z^{2} }$ and thus are independent.

Finally, let us bound the probability that $f$ and ${f}^{C,r}$ are different. Fix $x \in \Z^{2}$ and notice that, if $\d(x,C) \leq r$, Proposition~\ref{prop:finite_range_approximations_2} yields
\begin{equation}
\P \big( f(x) \neq {f}^{C,r}(x) \big) \leq \uc{c:approximation_environment}r^{-\beta+\frac{3}{2}}.
\end{equation}
If, on the other hand, $\d(x,C) > r$, then
\begin{equation}
\P \big( f(x) \neq {f}^{C,r}(x) \big) = \P \big( f(x) \neq f^{\d(x,C)}(x) \big) \leq \uc{c:approximation_environment}\d(x,C)^{-\beta+\frac{3}{2}}.
\end{equation}
This yields
\begin{equation}
\begin{split}
\P \big( f \neq {f}^{C,r} \big) & \leq \sum_{x \in \Z^{2}} \P \big( f(x) \neq {f}^{C,r}(x) \big) \\
& \leq 4\uc{c:approximation_environment}\big( |C|+\Per(C)r+r^{2} \big)r^{-\beta+\frac{3}{2}} + \sum_{x: \d(x,C) > r} \P \big( f(x) \neq {f}^{C,r}(x) \big) \\
& \leq 4\uc{c:approximation_environment}\big( |C|+\Per(C)r+r^{2} \big)r^{-\beta+\frac{3}{2}} + \sum_{x: \d(x,C) > r} \uc{c:approximation_environment}\d(x,C)^{-\beta+\frac{3}{2}} \\
& \leq 4\uc{c:approximation_environment}\big( |C|+\Per(C)r+r^{2} \big)r^{-\beta+\frac{3}{2}} + \sum_{k > r} 8\uc{c:approximation_environment}\big( k+\Per(C)\big) k^{-\beta+\frac{3}{2}} \\
& \leq \uc{c:bad_coupling}\big(wh+ (w+h)r+r^{2} \big)r^{-\beta+\frac{3}{2}},
\end{split}
\end{equation}
concluding the proof of the lemma.
\end{proof}

\begin{remark}
  \label{r:gff}
  The Gaussian free field also satisfies~\eqref{e:translation} and~\eqref{e:symmetry} by construction.
  The fact that it satisfies~\eqref{e:fkg} also follows from~\cite{pitt1982}, since this field is positively correlated.
  An analogous to Proposition~\ref{prop:finite_range_approximation} above for the $d$-dimensional discrete Gaussian Free Field $\varphi$ can be found in~\cite[Lemma 3.2]{duminil2020equality}.
  More precisely, there exist constants $c,C>0$ and, for every $L>0$, $L$-dependent approximations $\varphi^{L}$ that satisfy, for all $t > 0$,
  \begin{equation}
    \P\big( |\varphi_{x}-\varphi^{L}_{x}| \geq t \big) \leq Ce^{-ct^{2}L^{\frac{d-2}{2}}}.
  \end{equation}
  If one defines $f$ as in~\eqref{eq:field_approximation} but for the discrete Gaussian Free Field $\varphi$ and its finite-range approximations $\varphi^{L}$ in place of $g^{r}$, the proof of Lemma~\ref{l:error_decoupling} above can be combined with the above estimate to obtain the bound
  \begin{equation}
    \P \big( f(x) \neq f^{L}(x) \big) \leq
    \uc{c:approximation_environment}L^{-\frac{d-3}{4}},
  \end{equation}
  for every $x \in \Z^{d}$ and every $L \geq 1$.
  This implies that the Gaussian Free Field in dimension $d$ satisfies the Definition~\ref{def:decoupling} with decay rate
  \begin{equation}
    \varepsilon(w, h, L) =  \uc{c:bad_coupling}\big(wh+ (w+h)L+L^{2} \big)L^{-\frac{d-3}{4}}.
  \end{equation}
\end{remark}

\subsection{Confetti}
\label{s:confetti}
~

\par Let us now turn our attention to the confetti random environment described informally in Section \ref{s:intro}.
We start by providing a rigorous definition for it.

Consider a Poisson point process on $\mathbb{R}^2 \times \mathbb{R}_+ \times \{-1, 1\} \times [0,1]$ with intensity measure $\mu = \lambda(\d u) \otimes \nu(\d r)  \otimes \tfrac{1}{2}(\delta_{-1} + \delta_1) \otimes \lambda (\d w)$, where $\lambda$ denotes the Lebesgue measure and $\nu$ is a fixed reference distribution.
Recall from~\eqref{e:nu_decay} we are assuming that $\nu$ has tails that decay polynomially for some exponent $\alpha >2$.
We denote $\P$ the probability law of the Poisson point process defined on an appropriate sample space $\Omega$ whose elements are point-measures
\begin{equation}
  \label{e:PPP}
  \omega = \sum_{i \geq 0} \delta_{(u_i, r_i, d_i, w_{i})}.
\end{equation}
This space is endowed with the $\sigma$-algebra $\mathcal{F}$ generated by the evaluation maps $g_A(\omega) = \omega(A)$, for every Borel measurable set $A \subseteq \mathbb{R}^2 \times \mathbb{R}_+ \times \{-1, 1\} \times [0,1]$.

Each point in the realization of that Poisson point process has the form $(u_i, r_i, d_i, w_{i}) \in \mathbb{R}^2 \times \mathbb{R}_+ \times \{-1, 1\} \times [0,1]$, and is associated with a colored ball in the plane, where $u_i$ and $r_i$ determine respectively its center and radius, while $d_i \in \{-1, 1\}$ determines the color assigned to it, where $1$ is interpreted as the color red and $-1$ as the color blue. The variable $w_{i}$ will be used as a tie breaking rule.
Observe that, conditioned on the collection of centers $(u_{i})_{i \geq 0}$, the random variables $(r_{i})_{i \geq 0}$, $(d_{i})_{i \geq 0}$, and $(w_{i})_{i \geq 0}$ are independent and act as decorations of the planar Poisson point process.
The variable $d_i$ assumes the values $1$ or $-1$ with equal probability.

Recalling the representation $\omega = \sum_{i \geq 0} \delta_{(u_i, r_i, d_i, w_{i})}$, we say that
\begin{display}
  \label{e:cover}
  \emph{the ball $i$ covers a point $z \in \mathbb{R}^2$} if $\d(z, u_i) \leq r_i$.
\end{display}
Informally, every point $z \in \R^{2}$ will be assigned the same color as a ball uniformly chosen among all balls that cover it.
In case a certain point in the plane is not covered by any ball, it will be assigned the color $0$ (which may me interpreted as the color gray).
In order to construct this coloring precisely, we define the index of the largest weight of a ball covering a point $z \in \R^{2}$ as
\begin{equation}
  \label{e:cover_index}
  I_z = \underset{i}{\argmax} \big\{ w_i ; \, z \text{ is covered by } i \big\},
\end{equation}
where we implicitly assume that $I_z = \infty$ if no ball covers the point $z$.

Given a realization $\omega$, we now define the environment $f^{\omega}:\R^{2} \to \{-1,0,1\}$ via
\begin{equation}
  \label{e:coloring}
  f^{\omega}(z) =
  \begin{cases}
    d_{I_z} \qquad & \text{if $I_z < \infty$ and}\\
    0, \qquad & \text{otherwise}.
  \end{cases}
\end{equation}
We reader may consult Figure~\ref{f:confetti} for an illustration of the random environment we just constructed.

\begin{remark}
  Note that the variables $\{I_z\}_{z \in \R^{2}}$ act solely in order to define the color of a point covered by many boxes balls.
  Any other unambiguous tie-breaking strategy would also suffice for our results to hold.
\end{remark}

Once again~\eqref{e:translation} and~\eqref{e:symmetry} follows directly from the construction of the model, since the Poisson point process considered satisfies these two properties.

\bigskip

Let us now verify that $f$ satisfies the~\eqref{e:fkg} inequality.
It will be convenient for us to decompose the Poisson point process $\omega$ in \eqref{e:PPP} into two parts $\omega = \omega_{+} + \omega_{-}$ where
\begin{equation}
  \label{e:PPP_pm}
  \omega_{\pm} = \sum_{i \geq 0, d_i = \pm} \delta_{(u_i, r_i, d_i, w_{i})}.
\end{equation}
Each of the $\omega_{\pm}$ is a Poisson point processes with intensity $ \tfrac{1}{2} \lambda (\text{d} u) \otimes \nu (\text{d} r) \otimes (\delta_{-1} + \delta_{1}) \otimes \lambda (\d w)$ (notice that the variable $d_{i}$ is fixed).

We say an event $A$ is increasing if, for any $(\omega_{+}, \omega_{-}), (\tilde{\omega}_{+}, \tilde{\omega}_{-}) \in \Omega$ such that $\omega_{+} \leq \tilde{\omega}_{+}$ and $\omega_{-} \geq \tilde{\omega}_{-}$,
\begin{equation}\label{eq:increasing_confetti}
(\omega_{+},\omega_{-}) \in A \text{ implies } (\tilde{\omega}_{+},\tilde{\omega}_{-}) \in A.
\end{equation}

\begin{remark}\label{re:monotonicity_coloring}
Notice that, for any coloring $f$, we can write $f=f(\omega_{+}, \omega_{-})$, where $f$ is an increasing function of $\omega_{+}$ and decreasing on the variable $\omega_{-}$. In particular, an event $A$ depending on the coloring $f$ is monotone increasing in the sense of~\eqref{eq:increasing} if $f \leq \tilde{f}$ and $f \in A$ implies $\tilde{f} \in A$.
\end{remark}

The next proposition extends the FKG inequality for our setting, implying that increasing events are positively correlated.
\begin{proposition}
For any two increasing events $A$ and $B$,
\begin{equation}
  \label{e:fkg_confetti}
  \P(A \cap B) \geq \P(A) \P(B).
\end{equation}
\end{proposition}

\begin{proof}
Fix two increasing events $A$ and $B$. We start by applying the FKG inequality for Poisson point processes (in the variable $\omega_{-}$) in the conditional setting. This gives us, for each configuration $\omega_{+}$ fixed,
\begin{equation}
\E \big( \textbf{1}_{A}(\omega_{+},\omega_{-}) \textbf{1}_{B}(\omega_{+},\omega_{-}) \big| \omega_{+} \big) \geq  \E \big( \textbf{1}_{A}(\omega_{+},\omega_{-}) \big| \omega_{+} \big) \E \big( \textbf{1}_{B}(\omega_{+},\omega_{-}) \big| \omega_{+} \big).
\end{equation}
We further notice that the conditional expectations $\omega_{+} \mapsto \E\big( \textbf{1}_{A} \big| \omega_{+} \big)$ and $\omega_{+} \mapsto \E\big(\textbf{1}_{B} \big| \omega_{+} \big)$ are both monotone increasing, since $A$ and $B$ are monotone in each of the variables separately. Another application of the the usual FKG inequality yields
\begin{equation}
\begin{split}
\P(A \cap B) & \geq \E\big( \P(A|\omega_{+}) \P(B|\omega_{+}) \big) \geq \E\big( \P(A|\omega_{+}) \big) \E \big( \P(B|\omega_{+}) \big) = \P(A) \P(B),
\end{split}
\end{equation}
concluding the proof.
\end{proof}

\bigskip

The next lemma proves that the confetti random environment satisfies a decoupling condition stated in Definition~\ref{def:decoupling} with $\varepsilon$ as in~\eqref{e:polynomial}.

\nc{c:boolean_decoupling}
\begin{lemma}
  \label{l:error_decoupling_confetti}
  The Confetti coloring $f$ satisfies the decoupling condition from Definition~\ref{def:decoupling} with decay rate given by
  \begin{equation}
    \varepsilon(w, h, r) =  \uc{c:boolean_decoupling} \big( wh + (w + h) r + r^2 \big) r^{-\alpha},
  \end{equation}
  where $\uc{c:boolean_decoupling} > 0$ depends only on $\uc{c:decay_nu}$.
\end{lemma}
Note the similarity between the above and Lemma~\ref{l:error_decoupling}.
Before proceeding to the proof of the above lemma, we state a preliminary result.
Fixed a rectangle $R = [a, b] \times [c, d]$ and $r > 0$, we consider the set
of triples $(u, s, d)$, corresponding to balls with radius at least $r$ that intersect the box $R$:
\begin{equation}
\label{eq:large_intersection_event}
\mathcal{B}(R, r) = \Big\{ (u,s,d, w) \in \R^{2} \times \R_{+} \times \{-1, 1\} \times [0,1]; s \geq r \text{ and } B(u, s) \cap R \neq \emptyset \Big\}.
\end{equation}
The next lemma bounds the probability that a point in the Poisson process falls into $\mathcal{B}(R,r)$.

\nc{c:large_balls}
\begin{lemma}
  \label{l:large_ball_intersection}
  There exists $\uc{c:large_balls} > 0$ such that, for any $R = [a, b] \times [c, d]$ and $r \geq 1$,
  \begin{equation}
    \P \Big( \omega \big( \mathcal{B}(R, r) \big) > 0 \Big) \leq \uc{c:large_balls} \big( \Vol(R) + r \Per(R) + r^2 \big) r^{-\alpha},
  \end{equation}
  where $\omega$ denotes the Poisson point process~\eqref{e:PPP}, $\Vol(R)$ and $\Per(R)$ denote respectively the volume and perimeter of $R$.
\end{lemma}

\begin{proof}
Markov's inequality immediately implies
  \begin{equation}
    \P \Big( \omega \big( \mathcal{B}(R, r) \big) > 0 \Big) \leq \E \Big( \omega \big( \mathcal{B}(R, r) \big) \Big),
  \end{equation}
  so that it suffices to bound the above expectation.
  Now we have:
  \begin{equation}\label{eq:big_calculation_decouple}
    \begin{split}
      \E \Big( \omega( & \mathcal{B}(R, r)) \Big)
      = \int \nu \Big(  \big[ \max \big\{ r, \d( u, R ) \big\} , \infty \big) \Big) \d\lambda (u) \\
      & \overset{\eqref{e:nu_decay}}\leq \int \uc{c:decay_nu}  \max \big\{ r, \d( u, R ) \big\}^{-\alpha} \d\lambda (u)\\
      & = \int_{ \d( u, R ) \leq r} \uc{c:decay_nu} r^{-\alpha} \d\lambda (u) + \int_{\d(u, R) > r} \uc{c:decay_nu} \d(u, R)^{-\alpha} \d\lambda (u)\\
      & \leq \uc{c:decay_nu}\big( \Vol(R) + r\Per(R) +4r^{2} \big) r^{-\alpha} + \int_{r}^{\infty} \uc{c:decay_nu} \big( \Per(R) + 2 \pi s \big) s^{-\alpha} \d s \\
      & = \uc{c:decay_nu}\big( \Vol(R) + r \Per(R) +4r^{2}\big) r^{-\alpha} + \frac{\uc{c:decay_nu}}{\alpha - 1} \Per(R) r^{-\alpha + 1} + \frac{\uc{c:decay_nu}}{\alpha-2} \; 2 \pi \; r^{-\alpha + 2}\\
      & \leq \uc{c:large_balls}\big( \Vol(R) + r \Per(R) + r^2 \big) r^{-\alpha},
    \end{split}
  \end{equation}
  which concludes the proof.
\end{proof}

We are now ready to provide the proof of our decoupling inequality.
\begin{proof}[Proof of Lemma~\ref{l:error_decoupling_confetti}]
  Given a rectangle $D = [a, a + w] \times [b, b + h]$ and a radius $r > 0$, we define the Poisson Point process $\omega^{D, r}$ through
  \begin{equation}
    \label{e:omega_C_r}
    \omega^{D, r} := {\bf 1}_{\mathcal{B}(D, r)^c} \cdot \omega.
  \end{equation}
  This new process is essentially obtained by discarding every ball that touches $D$, but has radius larger or equal to $r$.
This allows us to introduce the desired field
  \begin{equation}
    \label{e:f_C_r}
    f^{D, r} := f^{\omega^{D,r}}
  \end{equation}
  and all we are left to do is to prove that the two conditions \eqref{e:likely_equal} and \eqref{e:finite_range} hold.

  The finite-range condition \eqref{e:finite_range} is a consequence  of the fact that all balls intersecting $D$ in $\omega^{D, r}$ have radius at most $r$.
  While the coupling condition \eqref{e:likely_equal} is implied by Lemma~\ref{l:large_ball_intersection} and the following calculation
  \begin{equation}
    \begin{split}
      \P \big( f^{C, r} \neq f \big) \quad
      & = \quad \P \Big( \omega \big( \mathcal{B}(D, r) > 0 \big) \Big)\\
      & \overset{\mathclap{\text{Lemma}~\ref{l:large_ball_intersection}}}\leq
        \quad \uc{c:large_balls} \big( \Vol(D) + \Per(D) r + r^2 \big) r^{-\alpha}
        \leq \varepsilon (w, h, r),
    \end{split}
  \end{equation}
  recalling that $h, w \geq 1$ and properly choosing $\uc{c:boolean_decoupling}$.
\end{proof}

\bibliographystyle{plain}
\bibliography{all}
\end{document}